\let\pl\cup 
\let\al\alpha
\let\be\beta

\let\si\sigma
\let\De\Delta
\def\ntbr(#1,#2,#3){{{\fam0 NB}(#1,#2;#3)}}
\def\sB{{\mathcal{B}}}
\def\sT{{\mathcal{T}}}
\let\isom\cong
\documentclass[12pt]{article}
\usepackage{graphicx}
\usepackage{amsmath, amsthm, amssymb}
\usepackage[shortlabels]{enumitem}
 \setlist[enumerate]{topsep=2pt,itemsep=2pt,partopsep=0pt,parsep=0pt}
\usepackage{graphicx}
\usepackage{color}
\usepackage{transparent}
\begin{document}
\newtheorem{theorem}{Theorem}[section]
\newtheorem{corollary}[theorem]{Corollary}
\newtheorem{observation}[theorem]{Observation}
\newtheorem{lemma}[theorem]{Lemma}
\newtheorem{conjecture}[theorem]{Conjecture}
 \theoremstyle{definition}
\newtheorem{claim}{Claim}
\newtheorem{remark}[theorem]{Remark}
%
%
 \let\reallabel\label
 \def\labelshow#1{\reallabel{#1}[#1]}
 \def\showlabels{\let\label\labelshow}
 \long\def\ignore#1{} 
 \long\def\noignore#1{#1}
 \long\def\redstuff#1{\color{red}#1\color{black}}
%
%
 \def\submitversion{%
	\let\altone\noignore \let\alttwo\ignore \let\altthree\ignore
 }%
 \def\detailversion{%
	\let\altone\ignore \let\alttwo\noignore \let\altthree\ignore
 }%
 \def\reddetailversion{%
	\let\altone\ignore \let\alttwo\redstuff \let\altthree\redstuff
 }%
 \detailversion

 \title{One-way infinite $2$-walks in planar graphs}
 \author{Daniel P.\ Biebighauser
  \footnote{%
  Department of Mathematics,
Concordia College, 901 8th St. S., Moorhead, Minnesota 56562.
 E-mail: \texttt{biebigha@cord.edu}. Supported by a sabbatical leave
from Concordia College.}
 \ \ and M.\ N.\ Ellingham
 \footnote{Department of Mathematics, 1326 Stevenson Center,
 Vanderbilt University, Nashville, Tennessee 37240.
 E-mail: \texttt{mark.ellingham@vanderbilt.edu}.
 Supported by National Security Agency grants H98230-04-1-0110 and
H98230-13-1-0233, and Simons Foundation award 245715.  The
United States Government is authorized to reproduce and distribute
reprints notwithstanding any copyright notation herein.}}
 \date{August 27, 2015}

 \maketitle

 \altthree{
 \noindent{\bf
 This version contains additional material (in red), giving some details
of omitted proofs for
 Lemma \ref{thm:forwardbridges},
 Theorem \ref{thm:triangprismlocfin}, and
 Theorem \ref{thm:triangprism}.}

 } 

 \begin{abstract}
 We prove that every $3$-connected $2$-indivisible infinite planar graph
has a $1$-way infinite $2$-walk.  (A graph is \emph{$2$-indivisible} if
deleting finitely many vertices leaves at most one infinite component,
and a \emph{$2$-walk} is a spanning walk using every vertex at most
twice.)  This improves a result of Timar, which assumed local
finiteness.  Our proofs use Tutte subgraphs, and allow us to also
provide other results when the graph is bipartite or an infinite analog
of a triangulation: then the prism over the graph has a spanning $1$-way
infinite path.
 \end{abstract}

\section{Introduction}\label{section:intro}

 For terms not defined in this paper, see \cite{bib:west}.  All
graphs are simple (having no loops or multiple edges) and may be
infinite, unless we explicitly state otherwise. 

 A \emph{cutset} in a graph $G$ is a set $S \subseteq V(G)$ such that $G
- S$ is disconnected.  A \emph{$k$-cut} is a cutset $S$ with $|S| = k$.
A graph is \emph{$k$-connected} if it has at least $k + 1$ vertices and
no cutset $S$ with $|S| < k$.  The \emph{connectivity} of a graph is the
smallest $k$ for which it is $k$-connected.

 The first major result on the existence of hamilton cycles in graphs
embedded in surfaces was by Whitney \cite{bib:whitney} in 1931, who
proved that every $4$-connected finite planar triangulation is
hamiltonian.  Tutte extended this to all $4$-connected finite planar
graphs in 1956 \cite{bib:tutte1}, and gave another proof in 1977
\cite{bib:tutte2}.  Tutte actually proved a more general result, using
subgraphs which have since been called ``Tutte subgraphs'' (defined in
Section \ref{section:stdpieces}).

 To extend these results to infinite graphs, one can look for infinite
spanning paths.
 We say that $v_1 v_2 v_3 \cdots$ is a \emph{$1$-way infinite path}, and
$\cdots v_{-2 }v_{-1} v_0 v_1 v_2 \cdots$ is a \emph{$2$-way infinite
path}, if each $v_i$ is a distinct vertex and consecutive vertices are
adjacent.

 If deleting finitely many vertices in an infinite graph leaves more
than one (two) infinite component(s), then the graph has no $1$-way
($2$-way) infinite spanning path.  Nash-Williams \cite{bib:nashwilliams}
defined a graph $G$ to be \emph{$k$-indivisible}, for a positive integer
$k$, if, for any finite $S \subseteq V(G)$, $G - S$ has at most $k - 1$
infinite components.  He conjectured \cite{bib:nashwilliams,
bib:nashwilliams2} that every $4$-connected $2$-indivisible
($3$-indivisible) infinite planar graph contains a $1$-way ($2$-way)
infinite spanning path.  The $1$-way infinite path conjecture was proved
by Dean, Thomas, and Yu \cite{bib:deanthomasyu} in 1996, and Xingxing Yu
established the $2$-way infinite path conjecture \cite{bib:yu2, bib:yu3,
bib:yu4, bib:yu5, bib:yu6}.

 For connectivity less than $4$, we must be more
flexible in the types of spanning subgraphs we wish to find, since there
are $3$-connected finite planar graphs with no hamilton path.
 Let $k$ be a positive integer.  A \emph{$k$-tree} is a spanning tree
with maximum degree at most $k$.  A \emph{$k$-walk} in a finite graph is
a closed spanning walk passing through each vertex at most $k$ times. 
In a finite graph, a $2$-tree is a hamilton path and a $1$-walk is a
hamilton cycle.

 Barnette \cite{bib:barnette} showed that every $3$-connected finite
planar graph contains a $3$-tree.  Gao and Richter \cite{bib:gaorichter}
later showed that every $3$-connected finite planar graph contains a
$2$-walk.
 Gao, Richter, and Yu \cite{bib:gaorichteryu,bib:gaorichteryu2} refined
this to give information about the locations of vertices included twice
in the $2$-walk.
  The $2$-walk results strengthen Barnette's result, because Jackson and
Wormald \cite{bib:jacksonwormald} showed that in a finite graph a
$k$-walk provides a $(k+1)$-tree.

 It is natural to ask if these results generalize to infinite graphs. 
In 1996, Jung \cite{bib:jung} proved that every locally finite
$3$-connected infinite plane graph with no vertex accumulation
points has a $3$-tree.  A graph is \emph{locally finite} if every vertex
has finite degree.

 A \emph{$1$-way ($2$-way) infinite walk} is a sequence
of vertices $v_1 v_2 v_3 \cdots$ ($\cdots v_{-2} v_{-1} v_0 v_1 v_2 \cdots$)
where consecutive vertices are adjacent.
 A \emph{$1$-way ($2$-way) infinite $k$-walk} is a $1$-way ($2$-way)
infinite spanning walk that passes through each vertex at most $k$
times.  A graph with a $1$-way or $2$-way infinite $k$-walk must be
infinite.

 In his doctoral dissertation, Timar proved the following two theorems. 

\begin{theorem}[Timar {\cite[Theorem II.2.5]{bib:timar}}]%
	\label{thm:timaroneway}
 Let $G$ be a locally finite $3$-connected $2$-indivisible infinite
planar graph.  Then $G$ has a $1$-way infinite $2$-walk.
 \end{theorem}

\begin{theorem}[Timar {\cite[Theorems II.3.3, II.4.3]{bib:timar}}]%
 	\label{thm:timartwoway}
 Let $G$ be a locally finite $3$-connected $3$-indivisible infinite
planar graph.  Then $G$ has a $2$-way infinite $2$-walk.
 \end{theorem}

 Our main result extends Theorem \ref{thm:timaroneway} by not requiring
local finiteness and by controlling the location of vertices used
twice.

 \begin{theorem}\label{thm:onewaytwowalkintro}
 Let $G$ be a $3$-connected $2$-indivisible infinite planar graph.  Then
$G$ has a $1$-way infinite $2$-walk for which any vertex used twice is in
a $3$-cut of $G$.
 \end{theorem}

 Our methods differ from those of Timar.  To avoid local finiteness we
use structural results, similar to those in \cite{bib:deanthomasyu}, for
graphs with infinite degree vertices.  To build the skeleton of the
$2$-walk and control the location of vertices that are used twice, we
use Tutte subgraphs in a way similar to \cite{bib:gaorichteryu}.
 We follow a systematic approach to using Tutte subgraphs that we have
developed in a survey paper \cite{bib:biebighauserellinghamtutte}, in
preparation.
 Our methods also provide other results when $G$ is bipartite or an
infinite analog of a triangulation (Theorems \ref{thm:bipartiteprism}
and \ref{thm:triangprism}).

 Timar \cite[Lemma I.2.14]{bib:timar} verified that Jackson and
Wormald's proof that  a $k$-walk provides a $(k+1)$-tree applies for
infinite graphs, and so Theorem \ref{thm:onewaytwowalkintro} allows us
to prove a result similar to Jung's.  We drop local finiteness, but add
$2$-indivisibility as both a hypothesis and a conclusion.

 \begin{corollary}
 Let $G$ be a $3$-connected $2$-indivisible infinite planar graph.  Then
$G$ has a $2$-indivisible $3$-tree.
 \end{corollary}


We also make the following natural conjecture.

\begin{conjecture}\label{thm:twowayconjecture}
Let $G$ be a $3$-connected $3$-indivisible infinite planar graph.  Then $G$ has a $2$-way infinite $2$-walk.
\end{conjecture}

 Proving Conjecture \ref{thm:twowayconjecture} would likely require
significant work along the lines of \cite{bib:yu2, bib:yu3, bib:yu4,
bib:yu5, bib:yu6}.

 Section \ref{section:definitions} includes some additional definitions
and lemmas, especially for connectivity.
 Section \ref{section:stdpieces} addresses Tutte subgraph techniques.  
 Section \ref{section:structural} discusses structural results for
$3$-connected $2$-indivisible infinite planar graphs.
 Section \ref{section:pathsandwalks} proves the main result, and
Section \ref{section:prisms} gives stronger theorems for
bipartite graphs and analogs of triangulations.  

\section{Definitions and Connectivity}\label{section:definitions}

 If $G$ is a connected finite plane graph, $X_G$ denotes the \emph{outer
walk} of $G$, the closed walk bounding the infinite face.
  If $G$ is $2$-connected, then we also call $X_G$ the \emph{outer
cycle} of $G$.
  We use $X_G$ to denote both a walk and a subgraph; if $G$ is
isomorphic to $K_2$ or $K_1$, then the subgraph $X_G$ is just $G$
itself.

 A \emph{$uv$-path} $P$ is a path from $u$ to $v$; $P^{-1}$ denotes the
reverse $vu$-path.
 If $P$ is a (possibly infinite) path and $x, y \in V(P)$, then $P[x,
y]$ denotes the subpath of $P$ from $x$ to $y$.
 Given a closed walk $W$ in a plane graph bounding an open disk (such as
a cycle or facial boundary walk), the subwalk $W[x,y]$ clockwise from
$x$ to $y$ is well-defined provided each of $x$ and $y$ occurs exactly
once on $W$.
 If $x = y$, then $W[x, y]$ or $P[x, y]$ means the single vertex $x =
y$.

 A \emph{block} is a $2$-connected graph or a graph isomorphic to $K_2$
or $K_1$.
 A \emph{block of a graph} is a maximal subgraph that is a block.  Every
graph has a unique decomposition into edge-disjoint blocks.  A block
isomorphic to $K_2$ is a \emph{trivial} block.  A vertex $v$ of a graph
is a \emph{cutvertex} if $\{v\}$ is a cutset of the graph.

 Let $G$ be a connected graph, and $n \ge 0$ an integer.
  Suppose that $G$ has finite blocks $B_1, B_2, \dots, B_n$
and vertices $b_0, b_1, b_2, \dots, b_{n-1}, b_n$ in $G$ such that
$G = b_0$ (if $n=0$) or
$b_0 \in V(B_1) - \{b_1\}$, $b_n \in V(B_n) - \{b_{n-1}\}$, $b_i \in
V(B_i) \cap V(B_{i+1})$ for $i = 1, 2, \dots, n-1$, and $G =
\bigcup_{i=1}^n B_i$.
 We say that $G$ is a \emph{chain of blocks} (some sources call this a
\emph{linear graph}) and that $(b_0, B_1, b_1, B_2, b_2, \dots, b_{n-1},
B_n, b_n)$ is a \emph{block-decomposition} of $G$.  In this case, $b_1,
b_2, \dots, b_{n-1}$ are precisely the cutvertices of $G$.  A chain of
blocks is a \emph{plane chain of blocks} if it is embedded in the plane
so that no block is embedded inside any other block.  In a plane chain
of blocks $G$, any internal face in a block will be a face of $G$.

 If $G$ has finite blocks $B_1, B_2, \dots$ and vertices $b_1, b_2,
\dots$ such that $b_i \in V(B_i) \cap V(B_{i+1})$ for every $i$, $G =
\bigcup_{i=1}^{\infty} B_i$, and $G$ is embedded in the plane so that no
block is embedded inside any other block, we say that $G$ is a
\emph{$1$-way infinite plane chain of finite blocks}.  (We will have no
need to specify an initial vertex $b_0$.)  In this case, we define $X_G$
to be the $2$-way infinite walk traversing the outer face of $G$ in the
clockwise direction.

 If $P$ is an $ab$-path in $G$, then the \emph{chain of blocks along $P$
in $G$} is the minimal union $K$ of blocks of $G$ that contains $P$ (or $K=a$ if $a=b$).  If $K=G$ we say that $G$ is a \emph{chain of
blocks along $P$}.
 We can write $K=(b_0, B_1, b_1, B_2, b_2, \dots, b_{n-1}, B_n, b_n)$
where $n \ge 0$, $b_0 = a$, and $b_n = b$; then $b_0, b_1, \dots, b_n$
are distinct vertices of $P$ and each $B_i$ contains an edge of $P$.

  A \emph{bridge of $H$}, or \emph{$H$-bridge}, in $G$ is either (a) an
edge of $E(G) - E(H)$ with both ends in $H$ (a \emph{trivial} bridge),
or (b) a component $C$ of $G - V(H)$ together with all of the edges with
one end in $C$ and the other in $H$.
  If $J$ is an $H$-bridge in $G$, then $E(J) \cap E(H) =
\emptyset$, $V(H) \cap V(J)$ is the set of \emph{attachments} of $J$
on $H$, and 
 $V(J)-V(H)$ is the set of \emph{internal vertices} of $J$ as an
$H$-bridge.  Let $A_G(H)$ be the set of attachments of all $H$-bridges
in $G$, or in other words, the vertices of $H$ incident with an edge of
$E(G)-E(H)$.

 We often use a property slightly weaker than being $k$-connected.
 Let $G$ be a graph, $k$ a positive integer, and $\emptyset \ne S
\subseteq V(G)$.
 We say $G$ is \emph{$k$-connected relative to $S$}, or \emph{$(k,
S)$-connected}, if for every $T \subseteq V(G)$ with $|T| < k$, every
component of $G - T$ contains at least one vertex of~$S$.  For a
subgraph $H$ of $G$ we say $G$ is \emph{$(k,H)$-connected} if it is
$(k, V(H))$-connected.
 Similar definitions have been used in earlier papers such as
\cite{bib:thomasyuzang,bib:yu2}, but ours differs in that we do not
require $G$ to be connected or $T$ to be a cutset.

 Our definition has a number of consequences that we use later.
 We omit the straightforward proofs.   Part \ref{ksz} may be regarded as
an alternative (perhaps more intuitive) definition; its proof uses
\ref{ksy} and Menger's Theorem, and it helps to prove later parts.

 \begin{lemma}\label{thm:ksconnected}
 Let $G$ be a graph, $k$ a positive integer, and $\emptyset \ne S
\subseteq V(G)$.

\begin{enumerate}[(a)]

 \item\label{ksx}
 If $S=V(G)$ then $G$ is always $(k,S)$-connected.  If $S \ne V(G)$ and
$G$ is $(k,S)$-connected then $|S| \ge k$.

 \item\label{ksy}
 $G$ is $(k,S)$-connected if and only if
 the graph obtained from $G$ by adding a new vertex $r$ adjacent to all
vertices of $S$ has no cutset $T$ with $|T|<k$ and $r \notin T$.

 \item\label{ksz}
 $G$ is $(k,S)$-connected if and only if
 for every $v \in V(G)-S$
there are $k$ paths, disjoint except at $v$, from $v$ to $S$ in $G$.

\item\label{ksa} If $G$ is $(k, S)$-connected, $G$ is a spanning
subgraph of $G'$, $1 \le k' \le k$, and $S \subseteq S'
\subseteq V(G)$, then $G'$ is $(k', S')$-connected.

\item\label{ksb} Adding or deleting edges with both ends in $S$ does not
affect whether or not $G$ is $(k, S)$-connected.

\item\label{ksc} Suppose $G$ is $k$-connected and $S \subseteq V(G)$
with $|S| \geq k$. Let $H$ be the union of $S$, zero or more $S$-bridges
in $G$, and an arbitrary set of edges with both ends in $S$. Then $H$ is
$(k, S)$-connected.  As a special case, $G$ is $(k, S)$-connected.

\item\label{ksd} Suppose $G$ is $(k, S)$-connected, and $H$ is a
subgraph of $G$.  Let $S_H = A_G(H) \cup (S \cap V(H))$.  Then $H$ is
$(k, S_H)$-connected.


 \item\label{kse} Suppose $G$ is $(k, S)$-connected and $H$ is a
subgraph of $G$ with $S \subseteq V(H)$.
 If $0 \le k' \le k$ and $H$ is $k'$-connected, then $G$ is also
$k'$-connected.
 Moreover, if $H \isom K_k$ and $V(H) \ne V(G)$ then $G$ is
$k$-connected.
 %
 %


\item\label{ksf}
 Construct $G'$ by adding to $G$ a set $R$ of new vertices, each
adjacent only to vertices in $R \cup S$.
 Then $G$ is $(k,S)$-connected if and only if $G'$ is $(k, R \cup
S)$-connected.

\end{enumerate}
\end{lemma}

 To prove a statement for $3$-connected finite planar graphs, one often
proves it for the following larger class of graphs.
 A \emph{circuit graph} is an ordered pair $(G, C)$ where $G$ is a
finite graph, $C$ is a cycle in $G$ that bounds a face in some plane
embedding of $G$, and $G$ is $(3,C)$-connected.  By Lemma
\ref{thm:ksconnected}\ref{kse}, $G$ is automatically $2$-connected.
 Frequently $C$ is the outer cycle of $G$.
 Barnette \cite{bib:barnette} originally defined a circuit graph as
the subgraph inside a cycle in a $3$-connected plane graph
--- this definition can be shown to be equivalent to ours using Lemma
\ref{thm:ksconnected}\ref{ksc}, \ref{kse} and \ref{ksf}.
 Also, by Lemma \ref{thm:ksconnected}\ref{ksc}, if $G$ is a $3$-connected
finite plane graph and $C$ is any facial cycle of $G$, then $(G, C)$ is
a circuit graph.

 It is convenient to define a finite plane graph $G$ to be a
\emph{circuit block} if either $G$ is an edge, or $(G, X_G)$ is a
circuit graph. A (possibly $1$-way infinite) \emph{plane chain of
circuit blocks} has the obvious meaning.

 Lemmas \ref{thm:circuitinduction} and \ref{thm:chaindelete4} below give
some useful inductive properties of circuit graphs, or more general
$(3,S)$-connected plane graphs.
 Lemma \ref{thm:circuitinduction} follows from Lemma
\ref{thm:ksconnected}\ref{ksa} and \ref{ksd}, and generalizes
\cite[Lemma 2]{bib:gaorichter}.  We use it frequently, often without
explicit mention.
 Lemma \ref{thm:chaindelete4}\ref{cdb} generalizes \cite[Lemma
3]{bib:gaorichter}.


 \begin{lemma}
 \label{thm:circuitinduction}
 Suppose $G$ is a $(3,S)$-connected plane graph, and $C$ is a cycle in
$G$ with no vertex of $S$ strictly inside $C$.
 If the subgraph $H$ of $G$ consisting of $C$ and everything inside $C$
is finite, then $(H, C)$ is a circuit graph.
 \end{lemma}


 \begin{lemma}
 \label{thm:chaindelete4}
 Suppose $P$ is a path in a finite connected plane graph $G$ and $P
\subseteq X_G$.

 \begin{enumerate}[(a)]
 \item\label{cda} If $G$ is $(3,P)$-connected then $G$ is a plane chain of
circuit blocks along $P$.

 \item\label{cdb} Suppose that $c \in V(X_G)-V(P)$.  If $G$ is $(3, P
\cup \{c\})$-connected then $G-c$ is a plane chain of circuit blocks
along $P$.
 \end{enumerate}
 \end{lemma}

 \begin{proof}
 Let $K$ be the chain of blocks in $G$ (for (a)) or $G-c$ (for (b))
along $P$.  The connectivity requirement means that in (a) there are no
$K$-bridges in $G$, and in (b) the only $(K \cup \{c\})$-bridges in $G$
are edges incident with $c$. By Lemma \ref{thm:circuitinduction} all
nontrivial blocks of $K$ are circuit graphs.  The results follow.
 \end{proof}

%
%



\section{Standard Pieces and Systems of Distinct Representatives}\label{section:stdpieces}

 To prove that every $4$-connected finite planar graph is
hamiltonian, Tutte used what are now known as ``Tutte
subgraphs.''  In this section, we describe some Tutte subgraph
results, including what we call ``standard pieces,'' that we will use
frequently.

 Let $X$ be a subgraph (usually given in advance) of a graph $G$,
and let $T$ be another subgraph (often a path or a cycle). 
Then $T$ is an \emph{$X$-Tutte subgraph} (or \emph{$X$-Tutte path} or
\emph{$X$-Tutte cycle}, if appropriate) of $G$ if: 

\begin{enumerate}[(i)]

\item every bridge of $T$ in $G$ has at most three attachments, and

\item every bridge of $T$ in $G$ that contains an edge of $X$ has at
most two attachments.

\end{enumerate}
 Sometimes no $X$ is given and only (i) holds; then
we simply say that $T$ is a
\emph{Tutte subgraph.}

 Our overall strategy for constructing a $1$-way infinite spanning
$2$-walk is to build a $1$-way infinite Tutte path $P$, and then detour
into the $P$-bridges to pick up all remaining vertices.  To build $P$ we
use a similar strategy to Dean, Thomas and Yu \cite{bib:deanthomasyu}. 
We build Tutte paths in finite parts of the graph, and then use the
argument of K\H{o}nig's Lemma to find finite paths converging to the
infinite path $P$.  To avoid using a vertex
more than twice when we detour into the $P$-bridges, we use an idea from
Gao, Richter and Yu \cite{bib:gaorichteryu}.
 We designate an entry point for each nontrivial bridge so that a vertex
is used as the entry point of at most one bridge.

 The entry points thus form a \emph{system of distinct representatives},
or \emph{SDR}, for the nontrivial $P$-bridges.
 Formally an SDR is an injective mapping from a set of subgraphs of a
graph $G$ to a set of vertices of $G$ so that each representative vertex
belongs to its subgraph.  We frequently refer to an SDR simply by its
set of representatives.  We never need to enter trivial bridges, so for
a subgraph $P$, an \emph{SDR of the $P$-bridges} means an SDR of the
nontrivial $P$-bridges


 Combining the ideas from \cite{bib:deanthomasyu} and
\cite{bib:gaorichteryu} is not straightforward; making these work
together is one of the main new contributions of this paper.
 First, finding the finite Tutte paths so that we also have an SDR of
their bridges can be complicated --- the most technical parts of the
proofs in Theorems \ref{thm:extendradialnet} and
\ref{thm:extendladdernet} are when we need to join Tutte subgraphs
together but also maintain an SDR of the bridges of their union.
 The general idea of $(3,S)$-connectedness helps here, allowing us to
use arguments that would be awkward to formulate just in terms of
circuit graphs.
 Second, when we use a K\H{o}nig's Lemma argument to get finite Tutte
paths converging to an infinite Tutte path $P$, in Theorem
\ref{thm:onewaypath}, we also need the SDRs for the finite paths to
converge to an SDR for $P$.  This requires a careful technical argument.
 Moreover, our methods allow us to obtain the stronger results in
Section \ref{section:prisms}.

 Throughout this paper, we use a general framework for arguments
involving Tutte subgraphs that we have developed in
\cite{bib:biebighauserellinghamtutte}; an early version appeared in
\cite{bib:dissertation}.
 Tutte subgraph arguments are often very technical and hard to follow;
our framework attempts to clarify them by emphasizing certain
fundamental ideas.
 Two key concepts are that Tutte subgraphs are constructed by piecing
together smaller Tutte subgraphs, and that many of these smaller Tutte
subgraphs are obtained using arguments that occur repeatedly.

  First we state a simple consequence of the definitions of a bridge,
Tutte subgraph, and SDR.  For a similar result (but without SDRs), see
\cite[(2.3)]{bib:thomasyu}.

 \begin{lemma}[Jigsaw Principle]\label{thm:jigsaw}
 Suppose $G$ is the edge-disjoint union of $G_1$, $G_2$, $\dots$, $G_k$. 
Suppose each $G_i$ has a subgraph $X_i$ and an $X_i$-Tutte subgraph
$T_i$ with an SDR $S_i$ of the $T_i$-bridges in $G_i$. 
 Suppose that $V(T_i) \cap V(T_j) = V(G_i) \cap V(G_j)$ and $S_i \cap
S_j = \emptyset$ for $i \ne j$.
 If $T = \bigcup_{i=1}^k T_i$,
 $X = \bigcup_{i=1}^k X_i$, and
 $S = \bigcup_{i=1}^k S_i$ then $T$ is an $X$-Tutte subgraph of $G$
with SDR $S$ of the $T$-bridges.
 Moreover, each $T$-bridge in $G$ is a $T_i$-bridge in $G_i$ for some $i$.
 \end{lemma}  

 We can think of a subgraph $G_i$ with its $X_i$-Tutte subgraph $T_i$
and SDR $S_i$ as a piece of a jigsaw puzzle; we can join pieces if they
``fit together'' correctly.  Usually at least one piece is found by
induction.  Other pieces are constructed using very standard arguments
(here, derived from Theorem \ref{thm:gry}) which form our
 \emph{standard piece lemmas}, or just \emph{standard pieces}.  Each
says that a graph with certain properties has a Tutte subgraph of a
certain type.

 We need three standard pieces involving SDRs, which we call SDR
Standard Piece $k$, or SDR SP$k$, for $k=1,2,3$.  As a mnemonic, $k$
denotes the number of components in the Tutte subgraph.
 Thomas and Yu gave related results without SDRs, combined into a single
theorem \cite[(2.4)]{bib:thomasyu}.  To deal with SDRs it helps to keep
the three situations separate; then the reader also knows exactly which
is being applied.
 We postpone the proofs until the end of this section.
 The figures show an $X$-Tutte subgraph $T$ having SDR $S$ with $X$
as dashed edges (green, if color is shown), $T$ as solid edges
and circled isolated vertices (red), and vertices known \emph{not} to be
in $S$ as solid vertices (blue).
 Solid vertices are used to make SDRs pairwise disjoint when
applying the Jigsaw Principle.


\begin{itemize}


\item {\bf SDR Standard Piece 1 (SDR SP1)}

{\bf Given:} A plane chain of circuit blocks $K = (a=b_0, B_1, b_1, B_2,
\dots, b_{n-1}, B_n, b_n=b)$ with $n \geq 0$, and $u \in V(X_K)$.

{\bf Then there exist:} An $X_K$-Tutte $ab$-path $P$ through $u$ in $K$
and an SDR $S$ of the $P$-bridges with $a \notin S$. 

\begin{figure}[h]
\centering
\def\svgwidth{5.5in}
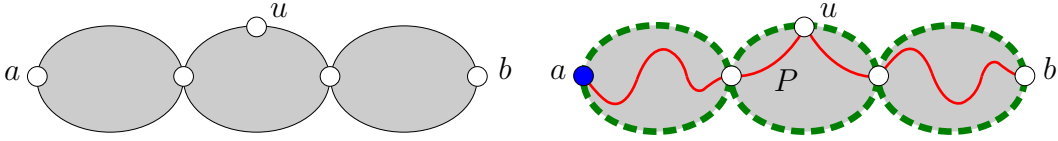
\caption{SDR Standard Piece 1}\label{fig:sdrsp1}
\end{figure}

\item {\bf SDR Standard Piece 2 (SDR SP2)}

{\bf Given:} A connected finite plane graph $K$ and $a, b, c \in V(X_K)$
such that (i) $X_K[a,b]$ is a path avoiding $c$, and (ii) $K - c$ is a plane
chain of circuit blocks $(a=b_0, B_1, b_1, B_2, \dots, \allowbreak
b_{n-1}, B_n, b_n = b)$ with $n \ge 0$.

{\bf Then there exist:} An $ab$-path $P$ avoiding $c$ such that $P \pl
\{c\}$ is an $X_K[a, b]$-Tutte subgraph of $K$, and an SDR $S$ of the
$(P \pl \{c\})$-bridges with $a, c \notin S$.

\begin{figure}[h]
\centering
\def\svgwidth{3.5in}
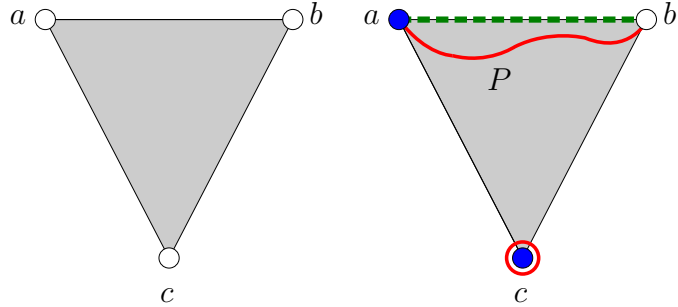
\caption{SDR Standard Piece 2}
\end{figure}

\item {\bf SDR Standard Piece 3 (SDR SP3)}

{\bf Given:} A connected finite plane graph $K$ and $a, b, c, d \in
V(X_K)$ such that (i) $c \ne d$, (ii) $X_K[a, b]$
is a path avoiding $c$ and $d$, and (iii) $K$ is $(3, X_K[a, b] \cup \{c,
d\})$-connected.  

{\bf Then there exist:}  An $ab$-path $P$ avoiding $c$ and $d$ such
that $P \pl \{c,d\}$ is an $X_K[a, b]$-Tutte subgraph of $K$, and for
each $x \in \{c,d\}$ an SDR $S$ of the $(P \pl \{c, d\})$-bridges with
$a, x \notin S$.

\begin{figure}[h]
\centering
\def\svgwidth{4.5in}
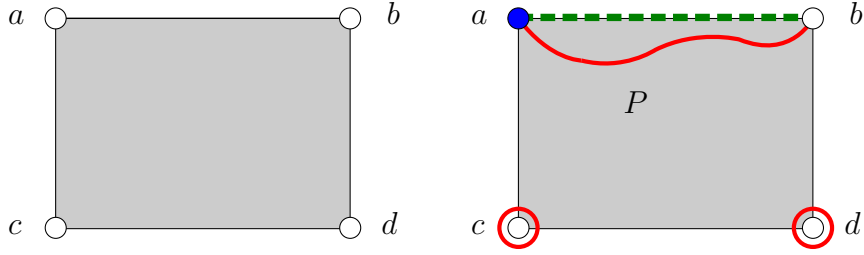
\caption{SDR Standard Piece 3 (one of $c$ or $d$ can be solid)}
\end{figure}

\end{itemize}

 To prove the SDR SPs we use Theorem \ref{thm:gry} below.
 It is \cite[Theorem 4]{bib:gaorichteryu,bib:gaorichteryu2} modified to
include the case
where $G$ is just an edge $xy$ (then we take $P=xy$ and $S =
\emptyset$).
 Note that a ``Tutte path'' in \cite{bib:gaorichteryu,
bib:gaorichteryu2} is what we call an ``$X_G$-Tutte path.''
 Corollary \ref{thm:gryedge} is used later.

 \begin{theorem}
 \label{thm:gry} 
 Let $G$ be a circuit block, $x, u \in V(X_G)$, $y \in
V(G)$ with $x \ne y$, and $v \in \{x, u\}$.  Then there is an
$X_G$-Tutte $xy$-path $P$ through $u$ and an SDR $S$ of the
$P$-bridges with $v \notin S$.
 \end{theorem}

 \begin{corollary}\label{thm:gryedge}
 Let $(G, X_G)$ be a circuit graph, $x \in V(X_G)$, $y \in
V(G)-\{x\}$, and $e \in E(X_G)$.  Then there is an $X_G$-Tutte $xy$-path
$P$ through $e$ and an SDR $S$ of the $P$-bridges with
$x \notin S$.
 \end{corollary}

 \begin{proof}
 Form $G'$ by subdividing the edge $e$ with a new
vertex $u$.  Then $(G', X_{G'})$ is a circuit graph by Lemma
\ref{thm:ksconnected}\ref{ksb} and \ref{ksf}.
Apply Theorem \ref{thm:gry} to
$(G', X_{G'})$, choosing $v = x$.
 \end{proof}

 \begin{proof}[Proof of SDR Standard Pieces 1, 2, and 3]
 For SDR SP1,
 if $a = b$, set $P = a =
b$ and $S = \emptyset$.  Otherwise, for each $B_i$, by Theorem
\ref{thm:gry}, we find an $X_{B_i}$-Tutte $b_{i-1}b_i$-path $P_i$ and an
SDR $S_i$ of the $P_i$-bridges such that $b_{i-1} \notin S_i$; if $u \in
V(B_i) - \{b_{i-1},b_i\}$, we choose $P_i$ to go through $u$.
 By the Jigsaw Principle, $P = \bigcup_{i=1}^n P_i$ and $S =
\bigcup_{i=1}^n S_i$ are as desired.

 \medskip
 For SDR SP2,
 apply SDR SP1 to $H=K-c$ to obtain an $ab$-path $X_{H}$-Tutte
path $P$ in $H$ and SDR $S$ of the $P$-bridges in $H$ with $a \notin
S$.
 Every nontrivial $P \pl \{c\}$-bridge $J$ in $K$ is a
$P$-bridge in $H$ unless it contains edges incident with $c$; in
that case $J$ must consist of a $P$-bridge $J'$ that uses an edge of
$X_{H}[b,a]$, and edges incident with $c$.  Then $J$ has three
attachments (two from $J'$, and $c$) and we may reassign the
representative of $J'$ to $J$.  Hence $P\pl\{c\}$ and $S$ (with some
reassignment) are as required.

 \medskip
 Finally, for SDR SP3,
 let $H = (b_0 = a, B_1, b_1, B_2, \dots, b_{n-1}, B_n, b_n = b)$ be the
plane chain of blocks along $X_K[a, b]$ in $K - \{c, d\}$.
 Since $K$ is $(3, X_K[a, b] \cup \{c, d\})$-connected, by applying
Lemma \ref{thm:circuitinduction} to the nontrivial blocks we see that
each $B_i$ is a circuit block.

 Every nontrivial $(H \cup \{c, d\})$-bridge has at most one attachment
in $H$, because $H$ is a chain of blocks, but at least three attachments
because $K$ is $(3, X_K[a, b] \cup \{c, d\})$-connected, so it must have
$c$, $d$ and exactly one vertex of $H$ as attachments.
 By planarity there can be at most one such bridge; if it exists, call
it $J$ and let $u$ be its attachment in $H$.

 By SDR SP1 there is an $X_H$-Tutte $ab$-path $P$ in $H$, through $u$ if
it exists, with an SDR $S'$ of the $P$-bridges such that $a \notin S'$.
 Consider the nontrivial $(P \cup \{c,d\})$-bridges in $K$.
 The only such bridge that can contain both $c$ and $d$ is $J$.  If $J$
exists, we choose $y$ as its representative, where $\{x,y\}=\{c,d\}$,
and set $S = S' \cup \{y\}$; otherwise set $S = S'$.
 The argument from the proof of SDR SP2 for nontrivial $P \cup \{c\}$
bridges applies here to nontrivial bridges with exactly one of $c$ or
$d$ as an attachment.
 Hence $P$ and $S$ are as desired.
 \end{proof}





 \begin{remark}\label{thm:2att3att}
 One idea from the proofs of SDR SP2 and SP3 is used often.
 If we delete a vertex $x$ from a graph, find a Tutte subgraph and an
SDR in what remains, and then add $x$ back, $x$ may become a new
attachment for some bridges.  We ensure that each such bridge previously
had only two attachments, so with $x$ there are still only three.
 Each such bridge already has a representative, so we do not need to use
$x$ as its representative.
 \end{remark}


\section{Structural Results}\label{section:structural}

In this section, we give some results about the structure of $3$-connected $2$-indivisible infinite planar graphs.  In these graphs, we find either an infinite plane chain of blocks or a structure called a \emph{net}.  Then, in Section \ref{section:pathsandwalks}, we use these structural results to build our $1$-way infinite $2$-walks.

If $G$ is a $2$-indivisible infinite plane graph and $C$ is a cycle in $G$, then $C$ divides the plane into two closed sets, exactly one of which contains finitely many vertices.  Let $I(C)$ (or $I_G(C)$) denote the subgraph of $G$ consisting of all vertices and edges of $G$ inside that closed set containing finitely many vertices.  Note that $C \subseteq I(C)$.  Dean, Thomas, and Yu \cite{bib:deanthomasyu} defined a \emph{net} in $G$ to be a sequence of cycles $N = (C_1, C_2, C_3, \dots)$ such that

\begin{enumerate}

\item $I(C_i)$ is a subgraph of $I(C_{i+1})$ for all $i = 1, 2, 3, \dots$,

\item $\bigcup_{i=1}^\infty I(C_i) = G$, and either

\item $C_1, C_2, C_3, \dots$ are pairwise disjoint, or

\item[3${}'$.] for every $i= 1, 2, 3, \dots$, the graph $C_i \cap C_{i+1}$ is a non-empty path, it is a subgraph of $C_{i+1} \cap C_{i+2}$, and no endpoint of $C_i \cap C_{i+1}$ is an endpoint of $C_{i+1} \cap C_{i+2}$.

\end{enumerate}
 If 3 holds we say that $N$ is a \emph{radial net}, and if 3${}'$ holds
we say that $N$ is a \emph{ladder net}.  A graph with a net is locally
finite, because for every vertex $v$ there is some $i$ such that $v$ and
all its neighbors belong to $I(C_i)$.

In \cite{bib:yu3}, Yu said that an infinite plane graph $G$ is \emph{nicely embedded} or is a \emph{nice embedding} if, for any cycle $C$ in $G$ for which $I(C)$, the finite side of $C$, is defined, $I(C)$ is contained in the closed disk bounded by $C$.  In a nice embedding, the intuitive idea of the ``inside'' of a cycle $C$ coincides with $I(C)$, which is why the notation $I(C)$ is used.  The following lemma is \cite[(2.1)]{bib:yu3}.

\begin{lemma}\label{thm:niceembedding} Let $G$ be an infinite plane graph, and suppose $G$ has a sequence of cycles $(C_1, C_2, C_3, \dots)$ such that $I(C_i)$ is a subgraph of $I(C_{i+1})$ for all $i = 1, 2, 3, \dots$, and $\bigcup_{i=1}^\infty I(C_i) = G$.  Then for any facial cycle $C$ of $G$, $G$ has a nice embedding in which $C$ is also a facial cycle.
\end{lemma}

By Lemma \ref{thm:niceembedding}, if a infinite plane graph has a net, then the graph has a nice embedding.  In this paper, we will always assume that such a graph is nicely embedded in the plane.

Let $N = (C_1, C_2, C_3, \dots)$ be a net in a $2$-indivisible plane
graph $G$.  The \emph{boundary} of $N$, denoted by $\partial N$, is the
graph $\bigcup_{i=1}^\infty (C_i \cap C_{i+1})$.  If $N$ is a radial
net, $\partial N = \emptyset$, and if $N$ is a ladder net, then
$\partial N$ is a $2$-way infinite path.  If $N$ is a ladder net, we
will assign an orientation $\overrightarrow{\partial N}$ to $\partial N$
such that $G - V(\partial N)$ is to the right of every edge in
$\overrightarrow{\partial N}$.  For $i = 1, 2, 3, \dots$, let $D_i$ be
the graph obtained from $C_i$ by deleting $C_i \cap C_{i+1}$ except its
endpoints.  If $N$ is a radial net, $D_i = C_i$, and if $N$ is a ladder
net, then $D_i$ is a path with both ends in $\partial N$ and otherwise
disjoint from $\partial N$.  If $N$ is a ladder net, and $C_1 - V(D_1)$
contains at least one vertex, let $D_0 = C_0$ be a subpath of $C_1 -
V(D_1)$.  Otherwise, set $D_0 = C_0 = \emptyset$ (but we will never see
this case in this paper).  We say that $N$ is \emph{tight} if

\begin{enumerate}

\item $I(C_1) = C_1$ if $N$ is a radial net,

\item $C_1 \cap C_2$ is either empty or contains at least one edge, and

\item for every $i = 1, 2, 3, \dots$, every $D_{i+1}$-bridge in $I(C_{i+1}) - V(I(C_i))$ has at most one attachment.

\end{enumerate}

If $N$ is a tight ladder net and every $D_1$-bridge in $I(C_1)- V(D_0)$
has at most one attachment, we say that $N$ is \emph{tight with respect
to $D_0$}.  Note that the definition of a tight ladder net in Dean,
Thomas, and Yu \cite{bib:deanthomasyu} also requires that $I(C_1) =
C_1$, which would imply that it is tight with respect to $D_0$; we will
not need this.

The following three lemmas are from \cite{bib:deanthomasyu}. The first immediately precedes their (1.2), and the others are their (2.1) and (2.2), respectively.

\begin{lemma}\label{thm:3connectedimplies2indivisible}
If $G$ is a $3$-connected planar graph and $X$ is any finite subset of $V(G)$, then $G - X$ has a finite number of components.
\end{lemma}
  
\begin{lemma}\label{thm:twoinfinite}
Let $G$ be a $2$-indivisible infinite plane graph such that the deletion of any finite set of vertices in $G$ results in a finite number of components.  Then $G$ has at most two vertices of infinite degree.
\end{lemma}

\begin{lemma}\label{thm:locallyfinitenet}
 Let $G$ be a locally finite $2$-connected $2$-indivisible infinite
plane graph.  Then $G$ has a net.
 \end{lemma}

How we find our $1$-way infinite Tutte path and SDR of its bridges in
every $3$-connected $2$-indivisible infinite plane graph will depend on
whether or not the graph contains a net, and, if so, what kind of net
the graph contains.  Let $G$ be a $3$-connected $2$-indivisible infinite
plane graph, and let $F$ be the set of vertices of infinite degree in
$G$.  By Lemmas \ref{thm:3connectedimplies2indivisible} and
\ref{thm:twoinfinite}, $|F| \leq 2$.  If $F = \emptyset$, then Lemma
\ref{thm:locallyfinitenet} guarantees that $G$ has a net.  If $|F| \geq
1$, then $G$ contains a spanning subgraph $H$ such that either

\begin{enumerate}

\item $H$ contains a ladder net (Lemma \ref{thm:infiniteblock}), or

\item $H$ is a $1$-way infinite plane chain of circuit blocks (Lemma
\ref{thm:infiniteblocks}).

\end{enumerate}

Thus there are essentially three types of subgraphs to consider: radial nets, ladder nets, and $1$-way infinite plane chains of circuit blocks.  We will deal with radial nets and ladder nets in Theorem \ref{thm:extendradialnet} and Theorem \ref{thm:extendladdernet}, respectively, and then combine the three cases in Theorem \ref{thm:onewaypath}. 

We first need two structural lemmas about $2$-indivisible infinite plane graphs containing at least one vertex of infinite degree.  The first is similar to \cite[(2.3)]{bib:deanthomasyu}.

\begin{lemma}\label{thm:infiniteblock}
Let $G$ be a $3$-connected $2$-indivisible infinite plane graph, let $F$ be the set of vertices of infinite degree in $G$, and assume that $|F| \geq 1$ and that $G - F$ has an infinite block.  Then there exists a $2$-connected $2$-indivisible infinite subgraph $H$ of $G$ such that

\begin{enumerate}

\item $H$ contains a ladder net $N$,

\item $F \subseteq V(\partial N)$,

\item $H$ is $(3, \partial N)$-connected, and

\item the only $H$-bridges in $G$ are edges incident with at least one vertex in $F$ (so $H$ is a spanning subgraph of $G$).

\end{enumerate}
\end{lemma}

\begin{proof} 

By Lemmas \ref{thm:3connectedimplies2indivisible} and \ref{thm:twoinfinite}, $|F| \leq 2$.  By \cite[(2.3)]{bib:deanthomasyu}, $G$ contains a $2$-connected $2$-indivisible infinite subgraph $H'$ such that $H'$ contains a net $N'$ and $F \subseteq V(\partial N')$. Since $F \ne \emptyset$, $N'$ must be a ladder net. Also by \cite[(2.3)]{bib:deanthomasyu}, every $H'$-bridge of $G$ is finite and has at most three attachments. Since $G$ is $3$-connected, any nontrivial $H'$-bridge in $G$ must have exactly three attachments. By the proof of \cite[(2.3)]{bib:deanthomasyu}, the attachments of any $H'$-bridge must be contained in $V(\partial N')$ and at most one of these attachments is in $V(\partial N') - F$. Thus any nontrivial $H'$-bridge must have exactly two attachments in $F$ and exactly one attachment in $V(\partial N') - F$.

If there are no nontrivial $H'$-bridges in $G$ (which must happen if $|F| = 1$ and may happen if $|F| = 2$), we set $H = H'$ and $N = N'$. Note that $H$ is $(3, \partial N)$-connected by Lemma \ref{thm:ksconnected}\ref{ksc}. Any $H'$-bridge is trivial and must have an attachment in $F$. Therefore $H$ and $N$ are as desired.

Otherwise, we may assume that there is at least one nontrivial $H'$-bridge. Then $|F| = 2$, so let $F = \{f_1, f_2\}$ and assume that $f_1$ comes before $f_2$ in $\overrightarrow{\partial N'}$. If $B$ is a nontrivial $H'$-bridge with $v$ as its attachment in $V(\partial N') - F$, we show that $v \in \partial N'[f_1, f_2]$.  If not, we may assume that $v$ comes before $f_1$ in $\overrightarrow{\partial N'}$. Then $L = B \cup \partial N'[v, f_2]$ is finite and contains a cycle which separates $f_1$ from $G - V(L)$, which contradicts the fact that $f_1$ has infinite degree.

By planarity, there is exactly one nontrivial $H'$-bridge $B$ with attachments $f_1$, $f_2$, and $v$, where $v \in \partial N'[f_1, f_2]$. Then $H = H' \cup B$ is a $2$-connected $2$-indivisible spanning subgraph of $G$. All $H$-bridges are trivial and have at least one attachment in $F$. Let $D$ denote the portion of $X_B$ from $f_1$ to $f_2$ that does not contain $v$. Then we modify the net $N' = (C_1', C_2', C_3', \dots)$ to form the net $N = (C_1, C_2, C_3, \dots)$ by setting $C_i = (C_i' - V(\partial N'[f_1, f_2])) \cup D$ for every $i \geq 1$. $N$ is a ladder net, $F \subseteq V(\partial N)$, and $H$ is $(3, \partial N)$-connected by Lemma \ref{thm:ksconnected}\ref{ksc}.
\end{proof}

The second structural lemma is proved in a similar way to that of Lemma \ref{thm:infiniteblock},
much as \cite[(2.4)]{bib:deanthomasyu} is similar to \cite[(2.3)]{bib:deanthomasyu}.

\begin{lemma}\label{thm:infiniteblocks}
Let $G$ be a $3$-connected $2$-indivisible infinite plane graph, let $F$ be the set of vertices of infinite degree in $G$, and assume that $|F| \geq 1$ and that $G - F$ has no infinite block.  Then $|F| = 2$, and there exists a connected subgraph $H$ of $G$ such that

\begin{enumerate}

\item $H$ is a $1$-way infinite plane chain of circuit blocks $(B_1, b_1, B_2, b_2, \dots)$,

\item $F \subseteq V(X_{B_1}) - \{b_1\}$,

\item the only $H$-bridges in $G$ are edges incident with at least one vertex in $F$ (so $H$ is a spanning subgraph of $G$).

\end{enumerate}
\end{lemma}

In each of the previous two lemmas, the proofs construct the desired
subgraph $H$ in such a way that we may assume that $H$ and $G$ are
nicely embedded in the plane.  The subgraph $H$ is locally finite in
both cases.  Also, if $|F| = 2$ and $F = \{f_1, f_2\}$, assuming that
$f_1$ comes before $f_2$ in $\overrightarrow{\partial N}$ ($X_H$), the
construction rules out any $H$-bridges in $G$ that are edges joining one
vertex in $F$ to a vertex in $V(\partial N[f_1, f_2]) - F$ (in
$V(X_H[f_1, f_2]) - F$, respectively).  We also remark that it follows
from these lemmas that every $3$-connected $2$-indivisible infinite
planar graph is countable.

The next two lemmas allow us to consider only tight nets instead of more
general nets.  The first lemma follows from the proof of
\cite[(2.6)]{bib:deanthomasyu} (in the first sentence of their proof,
they choose $C_1$ to be any facial cycle of $G$ with $u \in V(C_1)$,
which is why we are able to specify such a cycle $C$ in our Lemma
\ref{thm:tightradialnet}).
 The second lemma is similar to \cite[(2.5)]{bib:deanthomasyu} and Claim
1 in the proof of \cite[(3.6)]{bib:yu3}, and can be proven in a similar
way as those results --- a difference is that, with our slightly
modified definition of a tight ladder net, we can treat the first cycle
of the net in the same way as all of the other cycles.

\begin{lemma}\label{thm:tightradialnet}
Let $G$ be a $2$-connected $2$-indivisible plane graph with a radial net, let $u \in V(G)$, and let $C$ be any facial cycle of $G$ with $u \in V(C)$.  Then $G$ has a tight radial net $N = (C_1, C_2, C_3, \dots)$ with $C_1 = C$ (and hence $u$ is a vertex of $C_1$).
\end{lemma}

\begin{lemma}\label{thm:tightladdernet}
Let $G$ be a $2$-connected $2$-indivisible plane graph with a ladder net $N'$.  Let $u$ and $v$ be vertices of $\partial N'$, and suppose that $u$ comes before $v$ in $\overrightarrow{\partial N'}$ if $u$ and $v$ are distinct.  Then $G$ has a ladder net $N = (C_1, C_2, C_3, \dots)$ such that $u, v \in V(C_1) - V(D_1)$, $\partial N = \partial N'$, and $N$ is a tight ladder net with respect to $\partial N[u, v]$.
\end{lemma}


\section{Paths and $2$-walks}\label{section:pathsandwalks}

We will now find a $1$-way infinite Tutte path $P$ and an SDR of the
nontrivial $P$-bridges in every $3$-connected $2$-indivisible infinite
planar graph.  Using the SDR of the $P$-bridges, we then detour into
each bridge to build our $1$-way infinite $2$-walk.

Let $G$ be a $2$-indivisible plane graph with a net $N = (C_1, C_2, C_3,
\dots)$.  Then a $uv$-path $P$ in $G$ is a \emph{forward $uv$-path} if
whenever vertices $u, x, y, v$ occur on $P$ in this order, there is no
$i \in \{1, 2, 3, \dots\}$ such that $x \in V(C_{i+2}) - V(C_{i+1})$ and
$y \in V(C_i)$.  A forward path may move ``backward'' a little, but only
to a limited extent.  We will find forward Tutte paths in finite
portions of nets, and then use the fact that these paths are forward to
show (using a variation of K\H{o}nig's Lemma) that they converge
to a $1$-way infinite Tutte path.

We first focus on radial nets.  The following theorem is similar to
\cite[(3.4)]{bib:deanthomasyu}, but we assume that the graph is $(3,
C_1)$-connected, and we find an SDR of the $P$-bridges.

 \begin{theorem}\label{thm:extendradialnet}
 Let $G$ be a $2$-indivisible plane graph with a tight radial net $N =
(C_1, C_2, C_3, \dots)$ such that $G$ is $(3,C_1)$-connected, and let $u
\in V(C_1)$.
 Then for every $k \in \{1, 2, 3, \dots \}$, there exist a
vertex $v \in V(C_k)$, a forward $C_1$-Tutte $uv$-path $P$ in $I(C_k)$,
and an SDR $S$ of the $P$-bridges.
 \end{theorem}


 \begin{proof}
 The proof is by induction on $k$.  For $k = 1$, the one-vertex path $v
= u$ suffices, with $S = \{u\}$.

 So we can assume that $k > 1$ and that the statement holds for positive
integers less than $k$.  Let $H$ be the block of $I(C_k) - V(C_1)$
containing $C_k$, with the embedding inherited from $G$.  Since
$N$ is tight and $G$ is nicely embedded, $C_2$ bounds a face of $H$
containing $C_1$, and every $(H \cup C_1)$-bridge has at most one vertex
of attachment, called a \emph{tip}, in $C_2$.

 Let $t_1, t_2, \dots, t_n \in V(C_2)$ be all tips of $(H \cup
C_1)$-bridges of $I(C_k)$, listed in clockwise cyclic order on $C_2$. 
Since $G$ is $(3, C_1)$-connected, $n \geq 3$.  For $i = 1, 2, \dots,
n$, let $L_i$ be the union of all $(H \cup C_1)$-bridges that attach at
$t_i$.  The cycle $C_1$ has a collection of pairwise edge-disjoint
segments $\left\{ P(t_i) \right\}_{i=1}^n$ in clockwise cyclic order
such that, for every $i = 1, 2, \dots, n$, $P(t_i)$ contains $V(L_i)
\cap V(C_1)$ and the ends of $P(t_i)$ are in $V(L_i)$.  Since $n \ge 3$
the collection $\left\{ P(t_i) \right\}_{i=1}^n$ is well-defined.
 For each $i = 1, 2, \dots, n$, let $p_i$ and
$q_i$ be the endpoints of $P(t_i)$ so that $p_1, q_1, p_2,
q_2, \dots, p_n, q_n$ occur in $C_1$ in clockwise order.
 Let $u_1$ be the clockwise neighbor of $u$ on $C_1$.
 We may assume that $uu_1 \in C_1[q_n, q_1]$; hence $u \ne q_1$.

Now $(C_2, C_3, C_4, \dots)$ is a tight radial net in $H$.  Since $G$ is
$(3, C_1)$-connected, $H$ is $(3, C_2)$-connected by Lemma
\ref{thm:ksconnected}\ref{ksd}.  Therefore, by induction, there is a
vertex $v \in V(C_k)$, a forward $C_2$-Tutte $t_1v$-path $P'$ in $H$,
and an SDR $S'$ of the $P'$-bridges.
 Since $P'$ is $C_2$-Tutte and $G$ is $2$-connected, every
 $P'$-bridge $D$ containing an edge of $C_2$ has exactly two
attachments, which are on $C_2$; hence $V(D) \cap V(C_j) = \emptyset$ for
$j \ge 3$.

 We divide the $(H \cup C_1)$-bridges into three types:
(i) those with no tip, (ii) those with a tip on $P'$, and (iii) those
with a tip not on $P'$.  We will form collections of these bridges,
along with portions of $C_1$ that they span, to form subgraphs where we
can apply our standard pieces.
 If $t_i \in V(P')$, the subgraph $K_{t_i}$ collects together (is the
union of) $L_i$, $P(t_i)$, and every type (i) bridge with all
attachments in $P(t_i)$.
 If $t_i \in V(H) - V(P')$, then any bridge $J$ with tip $t_i$ is of
type (iii).  The $P'$-bridge $D$ in $H$ containing $t_i$ has exactly two
attachments, $c_D$ and $d_D$, which are in $C_2$.
 One of $c_D$ or $d_D$ is the representative of $D$ in $S'$.  Suppose
that $t_\ell, t_{\ell+1}, \dots, t_m$ are the tips in $V(D) - \{c_D,
d_D\}$.  The subgraph $K_D$ collects together $D, L_\ell, L_{\ell+1},
\dots, L_m, C_1[p_\ell, q_m]$, and every type (i) bridge with all
attachments in $C_1[p_\ell, q_m]$.

 For each $K_{t_i}$ with $i \ne 1$, $K_{t_i}$ is $(3, P(t_i) \cup
\{t_i\})$-connected by Lemma \ref{thm:ksconnected}\ref{ksd}, so by
Lemma \ref{thm:chaindelete4}\ref{cdb} we may apply SDR SP2 to obtain
a $p_iq_i$-path $P_{t_i}$ in $K_{t_i}$ avoiding $t_i$ such
that $P_{t_i} \pl \{t_i\}$ is a $P(t_i)$-Tutte subgraph and an SDR
$S_{t_i}$ of the $(P_{t_i} \pl \{t_i\})$-bridges with $q_i, t_i \notin
S_{t_i}$ (this is valid even if $p_i=q_i$).

 For each $K_D$, label $c_D$ and $d_D$ so that $c_D \in S'$, $d_D \notin
S'$.
 By Lemma \ref{thm:ksconnected}\ref{ksd},
 $K_D$ is $(3, C_1[p_\ell, q_m] \cup \{c_D, d_D\})$-connected.
 So SDR SP3 gives a $p_\ell q_m$-path $P_D$ in $K_D$ avoiding $c_D$ and
$d_D$ such that $P_D \pl \{c_D, d_D\}$ is a $C_1[p_\ell, q_m]$-Tutte
subgraph, and an SDR $S_D$ of the $(P_D \pl \{c_D, d_D\})$-bridges with
$q_m, d_D \notin S_D$.
 Then replacing $D$ and its representative $c_D$ by the $P_D$-bridges in
$K_D$ and $S_D$ does not use a representative more than once.

 For every segment $C_1[q_i, p_{i+1}]$, $i \ne n$ and $q_i \ne p_{i+1}$,
not already included in some $K_D$, we also collect $C_1[q_i, p_{i+1}]$
and every type (i) bridge with all attachments in $C_1[q_i, p_{i+1}]$ to
form a subgraph $K_i$.  By Lemmas \ref{thm:ksconnected}\ref{ksd} and
\ref{thm:chaindelete4}\ref{cda} we may apply SDR SP1 to $K_i$ to obtain
a $C_1[q_i, p_{i+1}]$-Tutte $q_ip_{i+1}$-path $P_i$ in $K_i$ and an SDR
$S_i$ of the $P_i$-bridges with $p_{i+1} \notin S_i$.


Let $P''$ be the union of $P_{t_i}$ for every $K_{t_i}$ considered
above, $P_D$ for every $K_D$ considered above, and $P_i$ for every $K_i$
considered above.  Let $S''$ be the union of every $S_{t_i}$, every
$S_D$, and every $S_i$.  Let $S_0'$ be the set of vertices $c_D$ for
every $K_D$ considered above.
 Then $P''$ is a $q_1q_n$-path, and by the Jigsaw Principle $P' \cup
P''$ is a $C_1[q_1, q_n]$-Tutte subgraph and $(S'-S_0') \cup S''$ is an
SDR of the $P' \cup P''$-bridges with $q_n \notin (S'-S_0') \cup S''$.

 Finally, we use the bridges with tip $t_1$ to find vertex-disjoint
$q_n u$- and $t_1q_1$-paths, and a corresponding SDR.
 Let $F$ be the union of $C_1[q_n, q_1]$ and every $(H \cup C_1)$-bridge
all of whose attachments on $C_1$ belong to $C_1[q_n, q_1]$.  
 Then $F$ contains $K_{t_1}$, and possibly some additional type (i)
bridges; $F$ also contains~$u$.
 By Lemma \ref{thm:ksconnected}\ref{ksd}, $F$ is $(3, C_1[q_1,q_n] \cup
\{t_1\})$-connected.
 So, by Lemma \ref{thm:chaindelete4}\ref{cdb}, $F - t_1$ is a plane
chain of circuit blocks $(b_0 = q_n, B_1, b_1, B_2, \dots, b_{m-1}, B_m,
b_m = q_1)$ with $m \ge 1$.
 Let $Y$ be the $q_nq_1$-path $X_{F-t_1}[q_n,q_1]$.

 Set $z$, $1 \leq z \leq m$, so that $u \in V(B_z) - \{b_z\}$ (recall
that $u \ne q_1$).  Let $w$ be the first vertex in $Y[b_{z-1}, q_1]$,
other than $b_{z-1}$, that is a neighbor of $t_1$ in $F$.  Since
$B_m-b_{m-1}$ has such a vertex, $w$ is defined.  Let $B_r$ be the block
such that $w \in V(B_r) - \{b_{r-1}\}$.  Then $z \leq r \leq m$.

 Let
 $H_1 = B_1 \cup B_2 \cup \cdots \cup B_{z-1}$,
 $H_2 = B_z \cup B_{z+1} \cup \cdots \cup B_r$, and
 $H_3 = B_{r+1} \cup B_{r+2} \cup \cdots \cup B_m$.
 Apply SDR SP1 to $H_1$ to find an $X_{H_1}$-Tutte $q_n b_{z-1}$-path
$Q_1$ in $H_1$ and an SDR $T_1$ of the $Q_1$-bridges with $b_{z-1}
\notin T_1$.
 Similarly, apply SDR SP1 to $H_3$ to find an $X_{H_3}$-Tutte $b_r
q_1$-path $Q_3$ in $H_3$ and an SDR $T_3$ of the $Q_3$-bridges with $q_1
\notin T_3$.

 Let $H_2' = H_2 \cup b_{z-1} w$, and $Z_2' =
X_{H_2'} = b_{z-1}w \cup Y[w, b_r] \cup C_1[b_{z-1}, b_r]$.
 By choice of $r$, $Z_2'$ is a cycle.
 Now $A_G(H_2)\subseteq V(Z_2')$, so by Lemma
\ref{thm:ksconnected}\ref{ksa}, \ref{ksb}, and \ref{ksd}, $H_2$ and
hence $H_2'$ are $(3, V(Z_2'))$-connected.
 Thus, $(H_2', Z_2')$ is a circuit graph.
 Apply Corollary \ref{thm:gryedge} to $H_2'$ to find a $Z_2'$-Tutte
$u b_r$-path $Q_2'$ through $b_{z-1}w$ and an SDR $T_2$ of the
$Q_2'$-bridges with $b_r \notin T_2$.
 Let $Q_2 = Q_2'-b_{z-1}w$ and $Z_2 = Z_2' - b_{z-1}w$, then $Q_2$ is
$Z_2$-Tutte, consisting of vertex-disjoint $b_{z-1} u$- and $w
b_r$-paths, and $T_2$ is an SDR of the $Q_2$-bridges in $H_2$.

 If any $Q_2$-bridge $B$ in $H_2$ has three attachments and contains any
edges of $X_{H_2}$, then these edges are in $Y[b_{z-1}, w]$.  By choice
of
$w$, no internal vertex of $B$ is a neighbor of $t_1$ in $F$.  Thus $B$
is a $Q_2$-bridge in $F$, still with only three attachments.
 Applying this and Remark \ref{thm:2att3att}, $Q = Q_1 \cup Q_2 \cup t_1
w \cup Q_3$ is $C_1[q_n,q_1]$-Tutte in $F$, consisting of
vertex-disjoint $q_n u$- and $t_1 q_1$-paths, with an SDR $T = T_1 \cup
T_2 \cup T_3$ of the $Q$-bridges such that $t_1, q_1 \notin T$.


Let $P = P' \cup P'' \cup Q$ and $S = (S'-S_0') \cup S'' \cup T$.  Then
$P$ is a $C_1$-Tutte $uv$-path in $I(C_k)$, and $S$ is an SDR of the
$P$-bridges.
 We must show that $P$ is a forward $uv$-path.  Suppose that $u,x,y,v$
are vertices in that order along $P$, with $x \in V(C_{i+2})$ for some
$i \in \{1, 2, \dots, k-2\}$.  Now $H \cap (P - V(P'))$ is contained in
the union of $P'$-bridges of $H$ that contain an edge of $C_2$, and
these bridges are vertex-disjoint from $C_{i+2}$, so $x \in V(P')$.  But
then $y \in V(P')$, so $y \notin V(C_i)$ because $P'$ is forward and
$V(P') \cap V(C_1) = \emptyset$.  Thus $P$ is a forward $uv$-path.
 \end{proof}

 Now let $G$ be a $2$-indivisible plane graph with a ladder net $N =
(C_1, C_2, C_3, \dots)$.  Suppose that $V(C_1) \ne V(D_1)$,
so that we may choose a subpath $C_0 = D_0$ of $C_1 - V(D_1)$.   For
each $i \geq 0$, suppose the path $D_i$ has ends $x_i$ and $y_i$,
 where $\dots, x_2, x_1, x_0, y_0, y_1, y_2, \dots$ occur
along $\partial N$ in this order.
 By construction, all of these vertices are distinct, except we might
have $x_0 = y_0$.  Set $I(C_0) = C_0 = D_0$.  See
Figure~\ref{fig:laddernet}.  Let $r$ and $s$ be integers such that $0
\leq r \leq s$.  The \emph{$(r, s)$-truncation of $G$ relative to $N$}
is the graph $G_{r, s}$ obtained from $I(C_s)$ by deleting the vertices
of $I(C_r) - V(D_r)$.  The following lemma is similar to Claim 2 in the
proof of \cite[(3.6)]{bib:yu3}, but we assume that $G$ is $(3, \partial
N)$-connected, and find an SDR of the $P$-bridges.

\begin{figure}[h]
\centering
\def\svgwidth{6in}
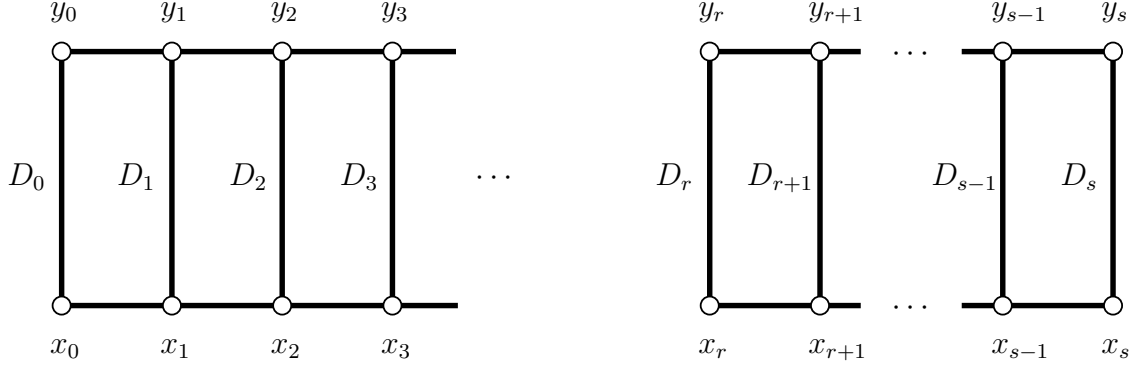
\caption{Ladder net notation --- the second graph is $G_{r, s}$}\label{fig:laddernet}
\end{figure}

 \begin{theorem}\label{thm:extendladdernet}
 Let $G$ be a $2$-indivisible plane graph with a ladder net $N =
(C_1, C_2, C_3, \dots)$ such that $G$ is $(3, \partial N)$-connected.
 Let $\dots, x_2, x_1, x_0, y_0, y_1, y_2, \dots$ be as above, and
assume that $N$ is tight with respect to $D_0 = \partial
N[x_0, y_0]$.
 Suppose that $0 \leq r \leq s$.
 Then there exist an $X_{G_{r, s}}[x_s, y_s]$-Tutte path $P$ in $G_{r,
s}$ from $x_r$ to $y_s$ if $s - r$ is even (to $x_s$ if $s - r$ is odd)
such that $\{x_r, x_{r+1}, \dots, x_s, y_r, y_{r+1}, \dots, y_s\}
\subseteq V(P)$ and $P$ is a forward path in $G$, and an SDR $S$ of the
$P$-bridges in $G_{r, s}$ such that $y_r \notin S$.  Symmetrically,
there is also such a path from $y_r$ to $x_s$ ($y_s$) and a
corresponding SDR of its bridges.
 \end{theorem}


 \begin{proof}
 The proof is by induction on $s - r$.  If $s = r$, then $P = D_s$ and
$S = \emptyset$ suffice.

 So assume that $r < s$.  Let $H=G_{r,s}$ and $H' = G_{r+1,s}$.
 By induction, there is an $X_{H'}[x_s, y_s]$-Tutte path $P'$ in $H'$
from $y_{r+1}$ to $y_s$ if $s - r$ is even ($x_s$ if $s - r$ is odd)
such that $\{x_{r+1}, x_{r+2}, \dots, x_s, y_{r+1}, y_{r+2}, \dots,
y_s\} \allowbreak \subseteq V(P')$ and $P'$ is a forward path in $G$
from $y_{r+1}$ to $y_s$ ($x_s$), and an SDR $S'$ of the $P'$-bridges
such that $x_{r+1} \notin S'$.
 By Lemma \ref{thm:ksconnected}\ref{ksd}, since $G$ is $(3, \partial
N)$-connected, $H$ is $(3, X_H)$-connected.  Since $X_H$ is a cycle,
$(H, X_H)$ is a circuit graph.

Since $N$ is tight, every $(H' \cup D_r)$-bridge in $G_{r, s}$ has at
most one attachment in $D_{r+1}$.  Let $t_1, t_2, \dots, t_n \in
V(D_{r+1})$ be all of the tips of $(H' \cup D_r)$-bridges in $G_{r, s}$,
listed in order in $X_{H'}[x_{r+1}, y_{r+1}]$.  Then $n \geq 2$, $t_1 =
x_{r+1}$, and $t_n = y_{r+1}$.
 So we may proceed similarly to the proof of Theorem
\ref{thm:extendradialnet}.  Define subpaths $P(t_i)$ of $D_r$ with ends
$p_i$ and $q_i$,
 so that $p_1 = x_r, q_1, p_2, q_2, \dots, p_n, q_n = y_r$ occur in this
order along $D_r$.
 Collect together subgraphs $K_{t_i}$ ($i \ne 1, n$), $K_D$, and $K_i$,
and use the SDR standard pieces to to find a $q_1 p_n$-path $P''$ such
that $P' \cup P''$ is a $D_r[q_1, p_n]$-Tutte subgraph and an SDR $S''$
of the $(P' \cup P'')$-bridges such that $p_n \notin S''$ (for each
piece, we ensure that the vertex in $D_r$ closest to $q_n$ is not in its
SDR).
 We must still deal with $K_{t_1}$ and $K_{t_n}$.

 Consider $K_{t_1}$.  Let $v$ be the clockwise neighbor of $t_1=x_{r+1}$ in
$X_H$.
 By Lemmas \ref{thm:ksconnected}\ref{ksd} and
\ref{thm:chaindelete4}\ref{cdb}, $K_{t_1} - t_1$ is a plane chain of
circuit blocks
 $A' = (a_0=v, A_1, a_1, A_2, \dots, a_{m-1}, A_m, a_m=q_1)$ with $m \ge
0$ along $X_H[v,q_1]$.
 Let $A=(a_\al, A_{\al+1}, a_{\al+1}, \dots, a_m)$ be the subchain of
$A'$ along $X_H[x_r,q_1]$ (with $a_\al$ as in $A'$).
 By SDR SP1 there is an $X_A$-Tutte $x_rq_1$-path $P_1$ through $a_\al$
in $A$ and an SDR $S'_1$ of the $P_1$-bridges with $q_1 \notin
S'_1$.
 If $\al > 0$ there is a nontrivial $(P_1 \cup \{t_1\})$-bridge
containing $t_1v$ with attachments $t_1$ and $a_\al$, so take
$t_1=x_{r+1}$ as its representative and set $S_1 = S'_1 \cup
\{x_{r+1}\}$; otherwise set $S_1 = S'_1$.  
 By Remark \ref{thm:2att3att}, $P_1 \cup \{t_1\}$ is an
$X_H[x_{r+1},q_1]$-Tutte subgraph of $K_{t_1}$; $S_1$ is an
SDR of the $(P_1 \cup \{t_1\})$-bridges with $q_1 \notin S_1$.


 Finally, consider $K_{t_n}$.   Let $w$ be the counterclockwise neighbor
of $t_n=y_{r+1}$ in $X_H$.
 By Lemmas \ref{thm:ksconnected}\ref{ksd} and
\ref{thm:chaindelete4}\ref{cdb}, $K_{t_n} - t_n$ is a plane chain of
circuit blocks $(b_0 = p_n, B_1, b_1, B_2, \dots, \allowbreak b_{z-1},
\allowbreak B_z, \allowbreak b_z = w)$ with $z \ge 0$.
 If $z = 0$, then $p_n=y_r=w$; let $P_n = y_r t_n$ and $S_n =
\emptyset$.
 If $z > 0$ let $B_\be$ be any block with $y_r \in V(B_\be)$.
 For each $j$, apply Theorem \ref{thm:gry} to find an $X_{B_j}$-Tutte
$b_{j-1} b_j$-path $R_j$ in $B_j$, through $y_r$ if $j = \be$, and an
SDR $U_j$ of the $R_j$-bridges such that $b_j \notin U_j$ if $j < \be$,
$y_r \notin U_j$ if $j=\be$, and $b_{j-1} \notin U_j$ if $j > \be$.
 Set $P_n = \left( \bigcup_{j=1}^z R_j \right) \cup wt_n$ and
$S_n = \bigcup_{j=1}^z U_j$.
 In either case, $P_n$ is an $X_H[p_n, y_{r+1}]$-Tutte
$p_n y_{r+1}$-path in $K_{t_n}$ (using Remark \ref{thm:2att3att} if
$z > 0$) and $S_n$ is an SDR of the $P_n$-bridges
with $y_r, y_{r+1} \notin S_n$ (but $y_{r+1}$ may already be in $S'$).

 Let $P = P' \cup P'' \cup P_1 \cup P_n$, and $S = S' \cup S'' \cup S_1
\cup S_n$.
 Then $P$ is an $X_H[x_s, y_s]$-Tutte path in $H=G_{r, s}$ from
$x_r$ to $y_s$ ($x_s$) such that $\{x_r, x_{r+1}, \dots, x_s, y_r,
y_{r+1}, \dots, y_s\} \subseteq V(P)$, and $S$ is an SDR of the
$P$-bridges such that $y_r \notin S$.  By an argument similar to the one
in Theorem \ref{thm:extendradialnet}, $P$ is a forward path from $x_r$
to $y_s$ ($x_s$).
 \end{proof}

 We will build our $1$-way infinite Tutte path from a sequence of finite
forward Tutte paths.  The following lemma will help to compare bridges
of different subgraphs in different graphs.
 \altone{%
 We omit the proof; the idea is that $R_1$ and $R_2$, and
$G_1$ and $G_2$, intersect $V(J)$, and the set of edges incident with a
vertex of $J$, in the same way.%
 } 

 \begin{lemma}\label{thm:forwardbridges}
 Let $G$ be a $2$-indivisible plane graph with a net $N = (C_1, C_2,
C_3, \dots)$.
 Let $i$ be a positive integer
 and $w \in V(D_{i+2})$.
 For $k = 1,2$ suppose that
 $I(C_{i+1}) \subseteq G_k \subseteq G$, and that
 $R_k \subseteq G_k$ is a forward path in $G$ from $u$ through $w$.
 Suppose that $R_1[u, w] = R_2[u, w]$ and $J \subseteq I(C_i)$.
 Then $J$ is an $R_1$-bridge in $G_1$ if and only if it is an
$R_2$-bridge in $G_2$, and $J$ has the same attachments in both cases.
 \end{lemma}

 \alttwo{%
 \begin{proof}
 For $J$ to be a bridge, it must be connected and have at least two
vertices.  Given this, $J$ is an $R_k$-bridge in $G_k$ with attachment
set $S$ if and only if (i) $E(J) \subseteq E(G_k) - E(R_k)$,
 (ii) $A_{G_k}(J) \subseteq V(R_k)$,
 (iii) $J-V(R_k)$ is connected or empty, and (iv) $V(J) \cap V(R_k) = S$.
 Since $R_k$ is forward, after $w$ it contains no vertex of $I(C_i)$ and
hence no vertex or edge of $J$, so we may replace $R_k$ in (i)--(iv)
with $R_k[u,w]$.
 Since every edge of $G$ incident with a vertex of $J$ belongs to
$I(C_{i+1})$ and $I(C_{i+1}) \subseteq G_k \subseteq G$, we may replace
$G_k$ in (i)--(iv) with $I(C_{i+1})$.
 But then the conditions are identical for $k = 1$ or $2$.
 \end{proof}

 } 

 We are now ready to find a $1$-way infinite Tutte path and an SDR of
its bridges in every $3$-connected $2$-indivisible infinite plane graph.
 We actually prove another similar result first, for graphs that have a
net but satisfy a somewhat weaker connectivity condition.
 The proof of this theorem is a variation of K\H{o}nig's Lemma and is
similar to \cite[(3.7) and (3.8)]{bib:deanthomasyu},
\cite[(3.5)]{bib:yu3}, and \cite[Lemma 5.2]{bib:yu4}.



 \begin{theorem}\label{thm:onewaypathlocfin}
 Let $G$ be a $2$-indivisible infinite plane graph with a net $N=(C_1,
C_2, C_3, \dots)$.
 If $N$ is a radial net, let $\De = C_1$ and $u = v \in V(C_1)$.
 If $N$ is a ladder net, let $\De = \partial N$ and $u, v \in \partial
N$.
 Assume that $G$ is $(3, \De)$-connected.
 Then there exist a $1$-way infinite $\De$-Tutte path $P$ in $G$ from
$u$ through $v$ and an SDR $S$ of the $P$-bridges.
 \end{theorem}

 \begin{proof}
 In either case, $G$ is locally finite since the neighbors of any
vertex belong to a finite graph $I(C_n)$ for large enough $n$. 
 Also, $G$ is $2$-connected by Lemma \ref{thm:ksconnected}\ref{kse},
taking $H$ to be $C_1$ when $N$ is a radial net, and the $2$-connected
subgraph $\bigcup_{i=1}^\infty C_i$ when $N$ is a ladder net.
 Therefore, by Lemma \ref{thm:tightradialnet} or
\ref{thm:tightladdernet}, we may assume that $N$ is a tight net, and
that if $N$ is a ladder net then $u, v \in V(C_1) - V(D_1)$ and $N$ is a
tight ladder net with respect to $D_0 = \partial N[u, v]$.

 Throughout this paper, a system of distinct representatives has been
represented by a set of vertices.  For our infinite limiting argument in
this proof, we must also keep track of the specific
assignment of vertices as representatives of different bridges.
 If $P$ is a subgraph of $G$, $\mathcal B$ is the set of nontrivial
$P$-bridges, and $\si: \mathcal B \rightarrow V(G)$ is an injection
such that, for every $B \in \mathcal B$, $\si(B) \in V(B) \cap V(P)$,
then we say that $\si$ is an \emph{assignment function} for $\mathcal
B$.  The range of $\si$ in $V(G)$ is an SDR $S$ of the $P$-bridges in
$\mathcal B$, and the existence of an SDR $S$ guarantees the existence
of $\si$.
 If we restrict the codomain of $\si$ to $S$, then $\si$ is a
bijection.

 If $N$ is a radial net, then by Theorem \ref{thm:extendradialnet} for
$n = 1, 2, 3, \dots$ we can find a forward $C_1$-Tutte path $Q_n$ in
$I(C_n)$ from $u$ to some vertex in $D_n$ (in this case, $D_n = C_n$)
and an SDR of the $Q_n$-bridges in $I(C_n)$.
 If $N$ is a ladder net, then by Theorem \ref{thm:extendladdernet} for
$n = 1, 2, 3, \dots$ we can find a forward $(\partial N \cap C_n)$-Tutte
path $Q_n$ in $G_{0,n}=I(C_n)$ from $u$ through $v$ to some vertex in
$D_n$ and an SDR of the $Q_n$-bridges in $I(C_n)$.
 In either case, each $Q_n$ is a forward $(\De \cap I(C_n))$-Tutte path
in $I(C_n)$ from $u$ through $v$.
 Let $\mathcal T_n$ be the collection of nontrivial $Q_n$-bridges in
$I(C_n)$, and let $\alpha_n: \mathcal T_n \rightarrow V(G)$ be an
assignment function for $\mathcal T_n$.

 We now construct an infinite sequence of paths and assignment functions
in $G$, which will converge to the path $P$ and to an assignment
function giving an SDR of the $P$-bridges, respectively.
 For subgraphs $H, K$ of a graph $F$, let $\ntbr(H, F, K)$ denote the set of nontrivial
$H$-bridges in $F$ that are subgraphs of $K$.

 Let $A_0 = \{1, 2, 3, \ldots\}$.
 Suppose $i \ge 1$ and we have an infinite set $A_{i-1}$ of positive
integers.
 Let $P_i$ be a path in $I(C_{i+2})$ from $u$ to a vertex of $D_{i+2}$
such that, for an infinite set $A_i' \subseteq A_{i-1}$ of values of
$n$, $Q_n$ has $P_i$ as an initial segment.  Such a $P_i$ exists because
$A_{i-1}$ is infinite, $I(C_{i+2})$ is finite and each $Q_n$, $n \in
A_{i-1}$, uses a vertex of $D_{i+2}$.  (There may be more than one such
path, but we fix just one as $P_i$.)
 Suppose $n \in A_i'$; necessarily $n \ge i+2$.
 Since $Q_n$ is a forward path from $u$ through $v$, $P_i$ is also a
forward path from $u$ through $v$.
 By Lemma \ref{thm:forwardbridges}, if we define $\sB_i = \ntbr(P_i,
I(C_{i+2}), I(C_i))$, then also $\sB_i = \ntbr(Q_n, I(C_n),
I(C_i)) \subseteq \sT_n$.
 Let $\si_i: \sB_i \rightarrow V(G)$ be an assignment
function for $\mathcal B_i$ such that, for an infinite set $A_i
\subseteq A_i'$ of values of $n$, $\alpha_n |_{\sB_i} =
\si_i$.
 Such a $\si_i$ exists because $A_i'$ is infinite and $\sB_i$
is finite.

 Constructing $A_i$ from $A_{i-1}$ in this way gives an infinite
sequence of sets $A_0 = \{1,2,3,\ldots\} \allowbreak \supseteq A_1
\supseteq A_2 \ldots$.  Suppose $i < j$ and choose $n \in A_j \subseteq
A_i$.  Since $P_i$ is an initial segment of $Q_n$ from $u$ to $D_{i+2}$
in $I(C_{i+2})$, and $P_j$ is an initial segment of $Q_n$ from $u$ to
$D_{j+2}$ in $I(C_{j+2})$, we have $P_i \subseteq P_j$.
 Hence $P = \bigcup_{i=1}^\infty P_i$ is a $1$-way infinite path; since
each $P_i$ is a forward path from $u$ through $v$, $P$ is also a
forward path from $u$ through $v$.
 Also, $\sB_i =
\ntbr(Q_n, I(C_n), I(C_i)) \subseteq \ntbr(Q_n, I(C_n), I(C_j)) =
\sB_j$.
 Hence $\sB_1 \subseteq \sB_2 \subseteq \sB_3 \subseteq \ldots \subseteq \sB=
\bigcup_{i=1}^\infty \sB_i$.
 Furthermore, since
 $\si_i = \al_n |_{\sB_i}$ and
 $\si_j = \al_n |_{\sB_j}$,
 $\si_j$ is an extension of $\si_i$.
 Hence we can define a function $\si : \sB \to V(G)$ where
 $\si|_{\sB_i} = \si_i$ for all $i$.
 Since each $\si_i$ is injective, $\si$ is injective, and its range $S$
is an SDR of $\sB$.


 Now we show that each $P$-bridge $J$ in $G$ is finite.
 Suppose not.  Then $J-V(P)$ is infinite, and by planarity it intersects
infinitely many $D_i$.  In particular, for some $i \ge 4$, $J-V(P)$
contains a path $M$ from $D_{i-3}$ to $D_i$; by terminating $M$ at its
first vertex in $D_{i}$, we may assume that $M \subseteq I(C_{i})$. 
 Choose $n \in A_{i}$.
 Then $M \subseteq I(C_{i}) - V(P) = I(C_{i}) - V(P_{i}) =
I(C_{i})-V(Q_n) \subseteq I(C_n)-V(Q_n)$, so
 $M$ is contained in some $Q_n$-bridge $J'$ in $I(C_n)$.
 Since $Q_n$ is a Tutte path in $I(C_n)$, $J'$ has at most three
attachments on $Q_n$.
 But since $M \subseteq J'$, $J'$ must have at least one attachment on
$Q_n$ in each of $D_{i-3}$, $D_{i-2}$, $D_{i-1}$, and $D_{i}$, a
contradiction.

 By Lemma \ref{thm:forwardbridges}, for each $i \ge 1$ we have
$\sB_i=\ntbr(P_i, I(C_{i+2}), I(C_i)) = \ntbr(P, G, I(C_i))$. 
Therefore, if $J \in \sB$ then $J \in \sB_i$ for some $i$ and hence $J$
is a $P$-bridge in $G$.  Conversely, if $J$ is a nontrivial $P$-bridge in
$G$ then, because $J$ is finite, $J \subseteq I(C_i)$ for some $i$, and
hence $J \in \sB_i \subseteq \sB$.  So $\sB$ is precisely the set of
nontrivial $P$-bridges in $G$.


 For each $P$-bridge $J$ in $G$, with $J \subseteq I(C_i)$ and $n \in
A_i$, Lemma \ref{thm:forwardbridges} gives that the number of
attachments of $J$ on $P$ in $G$ is the same as the number of
attachments of $J$ on $Q_n$ in $I(C_n)$, namely at most three, and at
most two if $J$ contains an edge of $\De$.

 We conclude that $P$ is a $\De$-Tutte path in $G$, and $S$ is an
SDR of the $P$-bridges.
 \end{proof}

 \begin{theorem}\label{thm:onewaypath}
 Let $G$ be a $3$-connected $2$-indivisible infinite plane graph, and
$F$ the set of vertices of infinite degree in $G$.  If $F = \emptyset$,
then $G$ has a net $N$; let $u = v \in V(G)$ be arbitrary if $N$ is a
radial net, and let $u, v \in V(\partial N)$ if $N$ is a ladder net.  If
$F = \{f_1\}$, let $u = v = f_1$.  If $F = \{f_1,f_2\}$, let $u=f_1$ and
$v=f_2$.  Then there exist a $1$-way infinite Tutte path $P$ in $G$ from
$u$ through $v$ and an SDR $S$ of the $P$-bridges.
 \end{theorem}

 \begin{proof}
 If $F = \emptyset$ then we can use Lemma \ref{thm:locallyfinitenet} to
find a net in $G$, Lemma \ref{thm:tightradialnet} to adjust a radial net
so that $u = v \in V(C_1)$, Lemma \ref{thm:ksconnected}\ref{ksc} to show
that $G$ is $(3,C_1)$- or $(3,\partial N)$-connected, and Theorem
\ref{thm:onewaypathlocfin} to find $P$ and $S$.

 So suppose that $|F| \ge 1$.
 Assume first that $G - F$ has no infinite block.
 By Lemma \ref{thm:infiniteblocks}, $|F| = 2$ and $G$ has a spanning
$1$-way infinite plane chain of finite circuit blocks $H = (B_1, b_1,
B_2, b_2, \dots)$ such that $F \subseteq V(X_{B_1}) - \{b_1\}$.
 If $i > 1$ then by Theorem \ref{thm:gry} we find an $X_{B_i}$-Tutte
$b_{i-1} b_i$-path $P_i$ in $B_i$ and an SDR $S_i$ of the $P_i$-bridges
in $B_i$ such that $b_{i-1} \notin S_i$.
 In $B_1$, by Theorem \ref{thm:gry}, we find an $X_{B_1}$-Tutte
$ub_1$-path $P_1$ in $B_1$ through $v$ and an SDR $S_1$ of the
$P_1$-bridges in $B_1$ such that $u \notin S_1$.
 Let $P = \bigcup_{i=1}^\infty P_i$ and $S = \bigcup_{i=1}^\infty S_i$,
then $P$ is an $X_H$-Tutte path from $u$ through $v$ in $H$ with an SDR
$S$ of the $P$-bridges in $H$.
 Now assume that $G-F$ has an infinite block.
 By Lemma \ref{thm:infiniteblock}, $G-F$ has a spanning subgraph $H$ with
a ladder net $N$ so that $H$ is $(3, \partial N)$-connected.
 Apply Theorem \ref{thm:onewaypathlocfin} to $H$ to obtain a $\partial
N$-Tutte path $P$ from $u$ through $v$ in $H$ with an SDR $S$ of the
$P$-bridges in $H$.
 In both cases, all edges of $G$ not in $H$ are incident with at least
one vertex of $F$, and applying Remark \ref{thm:2att3att} shows that $P$
is a Tutte path in $G$ and $S$ is still an SDR of the $P$-bridges in
$G$.
 \end{proof}

 We now find $1$-way infinite $2$-walks, using a $1$-way infinite Tutte
path $P$ provided by Theorem \ref{thm:onewaypathlocfin} or
\ref{thm:onewaypath} as the skeleton of each $2$-walk.
 To detour into each nontrivial $P$-bridge to visit the remaining
vertices, we use some results of Gao, Richter and Yu from the proof of
\cite[Theorem 6]{bib:gaorichteryu}.
 Lemma \ref{thm:finitetwowalk} is the more technical result from which
their Theorem 6 follows, modified to include the case of a trivial block.
 Gao, Richter and Yu also examine the structure of a bridge $L$ of a
Tutte subgraph in a circuit graph.
 Parts (a) and (b) of Lemma \ref{thm:bridgestructure} correspond to when
$L$ has three or two vertices of attachment, respectively;
 part (b) is just a special case of our Lemma \ref{thm:chaindelete4}(b).

 For a plane graph $G$, an \emph{internal $3$-cut} is a $3$-cut $A$ of
$G$ such that $G-A$ has a component vertex-disjoint from $X_G$.  Let
$N_G(v)$ denote the set of vertices adjacent to $v$ in $G$.

 \begin{lemma}[{\cite[proof of Theorem 6]{bib:gaorichteryu}}]
 \label{thm:finitetwowalk}
 Let $(G, X_G)$ be a circuit block, and let $x, y \in V(X_G)$ with $x
\ne y$.  Then $G$ contains a closed $2$-walk visiting $x$ and $y$
exactly once, such that every vertex visited twice is either in an
internal $3$-cut $A$ of $G$, or in a $2$-cut $A$ of $G$ with $A
\subseteq V(X_G[x, y])$ or $A \subseteq V(X_G[y, x])$.
 \end{lemma}

 \begin{lemma}[{\cite[proof of Theorem 6]{bib:gaorichteryu}}]
 \label{thm:bridgestructure}
 Let $L$ be a plane graph.

 \smallskip\noindent
 (a) Suppose $L$ is $(3,\{a,b,c\})$-connected, where $a,b,c$ are
distinct vertices each appearing once on $X_L$, in that clockwise order.
 Then $L-\{b,c\}$ is a plane chain of circuit blocks
 $K=(a=b_0, B_1, b_1, \ldots, b_{k-1}, B_k, b_k=d)$
 where $N_L(b) \subseteq V(X_K[a,d]) \cup \{c\}$ and $N_L(c) \subseteq
V(X_K[d,a]) \cup \{b\}$.

 \smallskip\noindent
 (b) Suppose $L$ is $(3,X_L[a,b])$-connected, where $a \ne b$.  Then
$L-b$ is a plane chain of circuit blocks $K=(a=b_0, B_1, b_1, \ldots,
b_{k-1}, B_k, b_k=d)$ where $d$ is the neighbor of $b$ in $X_L[a,b]$,
and $N_L(b) \subseteq V(X_K[d,a])$.
 \end{lemma}

 \begin{theorem}\label{thm:onewaytwowalklocfin}
 Let $G$ be a $2$-indivisible infinite plane graph with a net $N=(C_1,
C_2, C_3, \dots)$.
 If $N$ is a radial net, let $\De = C_1$.
 If $N$ is a ladder net, let $\De = \partial N$.
 Assume that $u \in V(\De)$ and that $G$ is $(3, \De)$-connected.
 Then $G$ contains a $1$-way infinite $2$-walk beginning at $u$ such
that every vertex used more than once belongs to a $2$- or $3$-cut of
$G$.
 \end{theorem}

 \begin{proof} By Theorem \ref{thm:onewaypathlocfin}, $G$ contains a
$1$-way infinite $\De$-Tutte path $P$ beginning at $u$ and an SDR $S$ of
the $P$-bridges.  Our walk $W$ will traverse $P$, beginning at $u$,
until we reach a vertex $a \in S$ that is a representative of a
nontrivial $P$-bridge $L$.

 If $L$ has three attachments $a,b,c$, then $L$ cannot contain an edge
of $\De$, so $L$ is $(3,\{a,b,c\})$-connected by Lemma
\ref{thm:ksconnected}\ref{ksd}.
 Then $L-\{b,c\}$ is a plane chain of circuit blocks
 $K=(a=b_0, B_1, b_1, \ldots, b_{k-1}, B_k, b_k=d)$
 as in Lemma \ref{thm:bridgestructure}(a), and we can apply Lemma
\ref{thm:finitetwowalk} to each block $B_i$ to get a $2$-walk using
$b_{i-1}$ and $b_i$ only once.  Combining these $2$-walks yields a
$2$-walk $W_a$ in $K$.  Each vertex used twice by $W_a$ is (i) some
$b_i$, $i \ge 1$, which is in a $3$-cut $\{b_i,b,c\}$ of $G$, (ii)  in
an internal $3$-cut of $K$, which is also a $3$-cut of $G$, or (iii) in
a $2$-cut $A$ of $K$ contained in either $X_K[a,d]$ or $X_K[d,a]$, so
that one of $A \cup \{b\}$ or $A \cup \{c\}$ is a $3$-cut of $G$.

 If $L$ has two attachments $a,b$ then necessarily $a, b \in
V(\De)$, and one of $X_L[a,b]$ or $X_L[b,a]$, call it $R$, is a subpath
of $\De$.
 By Lemma \ref{thm:ksconnected}\ref{ksd}, $L$ is $(3,R)$-connected.
 Then $L-b$ is a plane chain of circuit blocks $K$ as in Lemma
\ref{thm:bridgestructure}(b) (or its mirror image), and we can apply
Lemma \ref{thm:finitetwowalk} to each block, combining the resulting
$2$-walks to obtain a $2$-walk $W_a$.  Each vertex used twice by $W_a$
is (i) some $b_i$, $i \ge 1$, which is in a $2$-cut $\{b_i,b\}$ of $G$,
(ii) in an internal $3$-cut of $K$, which is also a $3$-cut of $G$, or
(iii) in a $2$-cut $A$ of $K$, which is either a $2$-cut of $G$ or such
that $A \cup \{b\}$ is a $3$-cut of $G$.
 
 Splicing $W_a$ into $P$ for every representative $a \in S$ gives the
required $2$-walk $W$.
 The vertices used twice by $W$ are the vertices used twice by each
$W_a$ and the vertices $a \in S$ themselves, each of which lies in a
$2$- or $3$-cut of $G$.
 \end{proof}

 The following result
can be proved in the same way as Theorem
\ref{thm:onewaytwowalklocfin}, using Theorem \ref{thm:onewaypath}
instead of Theorem \ref{thm:onewaypathlocfin}.
 In fact, the proof is simpler, because now $P$ has no bridges with two
attachments.
 Our main result, Theorem \ref{thm:onewaytwowalkintro}, follows
immediately from this.

 \begin{theorem}\label{thm:onewaytwowalk}
 Let $G$ be a $3$-connected $2$-indivisible infinite plane graph, and
let $F$ be the set of vertices of infinite degree in $G$.  If $F =
\emptyset$, then $G$ has a net $N$.  Let $u \in V(G)$ be arbitrary if
$N$ is a radial net, and let $u \in V(\partial N)$ if $N$ is a ladder
net.  Otherwise, let $u \in F$.  Then $G$ contains a $1$-way infinite
$2$-walk beginning at $u$ such that every vertex used more than once
belongs to a $3$-cut of $G$.
 \end{theorem}

\color{black} 

\section{Infinite Prisms}\label{section:prisms}

In this final section, we discuss spanning paths in prisms over infinite
planar graphs.  The \emph{prism} over a graph $G$ is the Cartesian
product $G \Box K_2$ of $G$ with the complete graph $K_2$.
 In \cite{bib:biebighauserellinghamprisms}, we showed that prisms over
bipartite circuit graphs and near-triangulations are hamiltonian.  A
\emph{near-triangulation} is a finite plane graph where every face is a
triangle, except for possibly the outer face, which is bounded by a
cycle.
 Here we extend these results to infinite graphs by showing that if $G$
is a $2$-indivisible infinite analog of a bipartite circuit graph or
near-triangulation then $G \Box K_2$ has a $1$-way infinite spanning
path.
 Tracing this path in the original graph $G$ gives a $1$-way infinite
$2$-walk.
 So for these classes of graphs we can strengthen the existence of a
$1$-way infinite $2$-walk in a somewhat different way from what we did
in Theorems \ref{thm:onewaytwowalklocfin} and \ref{thm:onewaytwowalk},
where we controlled the location of the vertices used twice.

 In the prism $G \Box K_2$, we may identify $G$ with one of its two
copies in the prism.  Let $v$ be a vertex in $G$.  In the prism, we let
$v$ denote the copy of the vertex in the graph that is identified with
$G$, and we let $v^*$ denote the other copy.  We use the same notation
for edges.  An edge of the form $vv^*$ is called a \emph{vertical} edge.
 Below we often take $u \in V(G)$ and find a path in $G \Box K_2$
beginning at $u$; then there is also a symmetric path beginning at
$u^*$.

 Our results for infinite bipartite graphs depend on the following
result for finite graphs.

 \begin{lemma}[{\cite[Theorem 2.4]{bib:biebighauserellinghamprisms}}]
 \label{thm:bipartitecircuit}
 Let $(G, X_G)$ be a bipartite circuit graph and let $u, v$ be two
distinct vertices in $X_G$.  Then there is a hamilton cycle in $G \Box
K_2$ that uses the vertical edges at $u$ and $v$.
 \end{lemma}

 \begin{theorem}
 \label{thm:bipartiteprismlocfin}
 Let $G$ be a $2$-indivisible infinite bipartite plane graph with a net
$N=(C_1, C_2, C_3, \dots)$.
 If $N$ is a radial net, let $\De = C_1$.
 If $N$ is a ladder net, let $\De = \partial N$.
 Assume that $u \in V(\De)$ and that $G$ is $(3, \De)$-connected.
 Then $G \Box K_2$ contains a $1$-way infinite spanning path beginning
at $u$.
 \end{theorem}

\begin{proof}
 By Theorem \ref{thm:onewaypathlocfin}, $G$ contains a $1$-way infinite
Tutte path $P = v_1 v_2 v_3v_4\ldots$ with $v_1 = u$.  Let $P' = v_1
v_1^* v_2^* v_2 v_3 v_3^* v_4^* v_4 \ldots$ be the $1$-way infinite
spanning path of $P \Box K_2$ starting at $u$ and using every vertical
edge.

 We modify $P'$ to detour into $L \Box K_2$ for each nontrivial
$P$-bridge $L$, as follows.  
 As in the proof of Theorem \ref{thm:onewaytwowalklocfin}, $L$ is
decomposed using Lemma \ref{thm:bridgestructure}.  Instead of using
Lemma \ref{thm:finitetwowalk} to obtain $2$-walks in circuit blocks that
we splice together, we use Lemma \ref{thm:bipartitecircuit} (extended to
allow blocks that are edges) to find hamilton cycles in the prisms over
circuit blocks, which we splice together at vertical edges (deleting
those vertical edges) and then splice into $P'$.
 \end{proof}

 The following result can be proved similarly, using Theorem
\ref{thm:onewaypath} instead of Theorem \ref{thm:onewaypathlocfin}.

 \begin{theorem}\label{thm:bipartiteprism}
 Let $G$ be a $3$-connected $2$-indivisible infinite bipartite plane
graph, and let $F$ be the set of vertices of infinite degree in $G$.  If
$F = \emptyset$, then $G$ has a net $N$.  Let $u \in V(G)$ be arbitrary
if $N$ is a radial net, and let $u \in V(\partial N)$ if $N$ is a ladder
net.  Otherwise, let $u \in F$.
 Then $G \Box K_2$ contains a $1$-way infinite spanning path beginning
at $u$.
 \end{theorem}

 We also want to give results for infinite versions of triangulations or
near-triangulations.  These are based on the following result for finite
graphs.

 \begin{lemma}[{\cite[Theorem 2.7]{bib:biebighauserellinghamprisms}}]
 \label{thm:circuittriangulation}
 Let $G$ be a finite near-triangulation and let $u, v$ be two
distinct vertices in $X_G$.  Then there is a hamilton cycle in $G \Box
K_2$ that uses the vertical edges at $u$ and $v$.
 \end{lemma}

 \begin{theorem}\label{thm:triangprismlocfin}
 Let $G$ be a connected $2$-indivisible infinite nicely embedded plane
graph with a net $N=(C_1, C_2, C_3, \dots)$.
 If $N$ is a radial net, let $\De = C_1$ and suppose that every face is
bounded by a triangle except perhaps one face bounded by $C_1$.
 If $N$ is a ladder net, let $\De = \partial N$ and suppose that every
finite walk bounding a face is a triangle.
 Then for every $u \in V(\De)$, $G \Box K_2$ contains a $1$-way infinite
spanning path beginning at $u$.
 \end{theorem}

 \begin{proof}
 \alttwo{%
 We show that $G$ is
$(3,\De)$-connected.  First notice that if $N$ is a radial net and
$v$ is any vertex, or if $N$ is a ladder net and $v \notin \partial N$,
then $v \in I(C_i) - V(C_i)$ for some $i$.  This follows from the
definition of a net.  Consequently, every edge incident with $v$ is
incident with two faces bounded by finite walks, since they are also
faces of the finite graph $I(C_i)$.

 Now suppose there is $T \subseteq V(G)$ with $|T| \le 2$ so that $G-T$
has a component $K$ with no vertex of $\De$.  Since $G$ is connected,
$|T| \ge 1$.

 Suppose first that $|T|=1$, with $T=\{t_1\}$.
 Then $t_1$ has neighbors $v_0$, $v_1$ consecutive in its rotation with
$v_0 \notin V(K)$ and $v_1 \in V(K)$.
 Let $f_0$ be the face containing $v_0 t_1 v_1$.
 If $N$ is a radial net, then $f_0$ is not bounded by $C_1$, since $v_1
\notin V(C_1)$.
 If $N$ is a ladder net, then $f_0$ is a face bounded by a finite walk
since $v_1 \notin \partial N$.
 In either case, $f_0$ is a triangle.
 Since $t_1$ is the only place where we can change between vertices in
$K$ and not in $K$, $f_0 = t_1 \ldots v_0 t_1 v_1 \ldots t_1$.
 Hence, $f_0$ has length at least $4$, and has a repeated vertex, either
of which is a contradication.

 Now suppose that $|T|=2$, with $T=\{t_1,t_2\}$.
 Then there is a sequence $v_0, v_1, \ldots, v_k, v_{k+1}$ of neighbors
of $t_1$ in clockwise order with $k \ge 1$, $v_1, v_2, \ldots, v_k \in
K$, and $v_0, v_{k+1} \notin K$; these vertices are distinct except
possibly $v_0 = v_{k+1}$.  
 Let $f_0$ be the face containing $v_0 t_1 v_1$ and $f_k$ the face
containing $v_k t_1 v_{k+1}$.
 By the same reasoning as in the case $|T|=1$, $f_0$ is a triangle.
 Now since $t_1$ and $t_2$ are the only places we can change between
vertices in $K$ and not in $K$, $f_0 = t_i \ldots v_0 t_1 v_1 \ldots
t_i$ where possibly $v_0 = t_i$, but $v_1 \ne t_i$. 
 Since $f_0$ is a triangle we must have $v_0 = t_2$ and $f_0 = t_2 t_1
v_1 t_2$.
 By similar reasoning, $v_{k+1}=t_2$ and $f_k = t_2 t_1 v_k t_2$.  Since
we cannot have two edges $t_1 t_2$, the neighbors of $t_1$ must be
exactly $t_2, v_1, v_2, \ldots v_k$.  Since, from above, $\{t_2\}$
cannot isolate a component with no vertex of $\De$, $t_1 \in V(\De)$. 
But $\De$ is $2$-regular, which means one of $v_1, v_2, \ldots, v_k$ is
in $\De$, a contradiction.
 
 Hence no such $T$ exists, and $G$ is $(3,\De)$-connected.

 }
 \altone{%
 Using arguments similar to the standard ones used to
show that finite triangulations are $3$-connected, we can show that $G$
is $(3,\De)$-connected; we omit the details.%
 } 
 We now proceed as in the proof of Theorem
\ref{thm:bipartiteprismlocfin}, using Theorem
\ref{thm:onewaypathlocfin}, with Lemma \ref{thm:circuittriangulation}
instead of Lemma \ref{thm:bipartitecircuit}.
 This works because our conditions on faces guarantee that every circuit
graph that we consider as a subgraph of a bridge $L$ of the $1$-way
infinite walk $P$ is actually a near-triangulation.
 \end{proof}

 We can also prove the following, using Theorem \ref{thm:onewaypath} and
Lemma \ref{thm:circuittriangulation}.  We say a face has \emph{bounded
extent} if it is bounded in the metric space sense.

 \begin{theorem}\label{thm:triangprism}
 Let $G$ be a $3$-connected $2$-indivisible infinite plane graph with a
nice embedding in which every
 finite walk bounding a face is a triangle.
 Let $F$ be the set of vertices of infinite degree in $G$.  If $F =
\emptyset$, then $G$ has a net $N$.  Let $u \in V(G)$ be arbitrary if
$N$ is a radial net, and let $u \in V(\partial N)$ if $N$ is a ladder
net.  Otherwise, let $u \in F$.
 Then $G \Box K_2$ contains a $1$-way infinite spanning path beginning
at $u$.

 As a special case, suppose $G$ is a locally finite $2$-indivisible
infinite plane graph in which every edge is incident with two
faces of bounded extent, each bounded by a triangle.
 Then for every $u \in V(G)$,
 $G \Box K_2$ contains a $1$-way infinite spanning path beginning
at $u$.
 \end{theorem}

 \alttwo{
 \begin{proof}
 In the special case we can prove that the graph is $3$-connected by the
same reasoning we used in the previous theorem.  Therefore there is a
net, which must be a radial net, because for a ladder net $N$ the edges
of $\partial N$ fail the condition.  The embedding must be nice because
the side of any cycle with finitely many vertices partitions into a
finite number of triangles, each of which is bounded in the metric space
sense.
 \end{proof}

 }


\begin{thebibliography}{99}

\bibitem{bib:barnette} David Barnette, Trees in polyhedral graphs, \emph{Canad. J. Math.} {\bf 18} (1966), 731--736.

\bibitem{bib:dissertation} Daniel P. Biebighauser, Generalizations of hamiltonicity for embedded graphs, doctoral dissertation, Vanderbilt University (2006).

\bibitem{bib:biebighauserellinghamprisms} Daniel P. Biebighauser and M. N. Ellingham, Prism-hamiltonicity of triangulations, \emph{J. Graph Theory} {\bf 57} No. 3 (2008), 181--197.

\bibitem{bib:biebighauserellinghamtutte} Daniel P. Biebighauser and M. N. Ellingham, Tutte paths and cycles, in preparation.

\bibitem{bib:chibanishizeki} Norishige Chiba and Takao Nishizeki, A theorem on paths in planar graphs, \emph{J. Graph Theory} {\bf 10} No. 4 (1986), 449--450.

\bibitem{bib:deanthomasyu} Nathaniel Dean, Robin Thomas, and Xingxing Yu, Spanning paths in infinite planar graphs, \emph{J. Graph Theory} {\bf 23} No. 2 (1996), 163--174.

\bibitem{bib:gaorichter} Zhicheng Gao and R. Bruce Richter, $2$-walks in circuit graphs, \emph{J. Combin. Theory Ser. B} {\bf 62} No. 2 (1994), 259--267.

\bibitem{bib:gaorichteryu} Zhicheng Gao, R. Bruce Richter, and Xingxing Yu, $2$-walks in $3$-connected planar graphs, \emph{Australas. J. Combin.} {\bf 11} (1995), 117--122.

\bibitem{bib:gaorichteryu2} Zhicheng Gao, R. Bruce Richter, and Xingxing Yu, Erratum to: $2$-walks in $3$-connected planar graphs, \emph{Australas. J. Combin.} {\bf 36} (2006), 315--316.

\bibitem{bib:jacksonwormald} B. Jackson and N. C. Wormald, {$k$}-walks of graphs, {\it Australas. J. Combin.} {\bf 2} (1990), 135--146.

\bibitem{bib:jung} Hwan-Ok Jung, On spanning 3-trees in infinite 3-connected planar graphs, {\it Commun. Korean Math. Soc.} {\bf 11} No. 1 (1996), 1--21.

\bibitem{bib:nashwilliams} C. St. J. A. Nash-Williams, Hamiltonian lines in infinite graphs with few vertices of small valency, \emph{Aequationes Math.} {\bf 7} (1971), 59--81.

\bibitem{bib:nashwilliams2} C. St. J. A. Nash-Williams, Unexplored and semi-explored territories in graph theory, in ``New directions in the theory of graphs (Proc. Third Ann Arbor Conf., Univ. Michigan, Ann Arbor, Mich., 1971),'' Academic Press, New York (1973), 149--186.

\bibitem{bib:thomasyu} Robin Thomas and Xingxing Yu, 4-connected projective-planar graphs are hamiltonian, \emph{J. Combin. Theory Ser. B} {\bf 62} No. 1 (1994), 114--132.

\bibitem{bib:thomasyuzang}
 Robin Thomas, Xingxing Yu and Wenan Zang,
 Hamilton paths in toroidal graphs,
 \emph{J. Combin. Theory Ser. B} {\bf 94} (2005) 214--236.

\bibitem{bib:thomassen} Carsten Thomassen, A theorem on paths in planar graphs, \emph{J. Graph Theory} {\bf 7} No. 2 (1983), 169--176.

\bibitem{bib:timar} C. Cary Timar, Spanning walks in infinite planar graphs, doctoral dissertation, Vanderbilt University (1999).

\bibitem{bib:tutte1} W. T. Tutte, A theorem on planar graphs, \emph{Trans. Amer. Math. Soc.} {\bf 82} (1956), 99--116.

\bibitem{bib:tutte2} W. T. Tutte, Bridges and hamiltonian circuits in planar graphs, \emph{Aequationes Math.} {\bf 15} No. 1 (1977), 1--33.

\bibitem{bib:west} Douglas B. West, ``Introduction to graph theory,'' Prentice Hall, Upper Saddle River, NJ (2001).

\bibitem{bib:whitney} H. Whitney, A theorem on graphs, \emph{Ann. Math.} {\bf 32} (1931), 378--390.

\bibitem{bib:yu2} Xingxing Yu, Infinite paths in planar graphs, I, graphs with radial nets, \emph{J. Graph Theory} {\bf 47} No. 2 (2004), 147--162.

\bibitem{bib:yu3} Xingxing Yu, Infinite paths in planar graphs, II, structures and ladder nets, \emph{J. Graph Theory} {\bf 48} No. 4 (2005), 247--266.

\bibitem{bib:yu4} Xingxing Yu, Infinite paths in planar graphs, III, $1$-way infinite paths, \emph{J. Graph Theory} {\bf 51} No. 3 (2006), 175--198.

\bibitem{bib:yu5} Xingxing Yu, Infinite paths in planar graphs, IV, dividing cycles, \emph{J. Graph Theory} {\bf 53} No. 3 (2006), 173--195.

\bibitem{bib:yu6} Xingxing Yu, Infinite paths in planar graphs, V, 3-indivisible graphs, \emph{J. Graph Theory} {\bf 57} No. 4 (2008), 275--312.

\end{thebibliography}
\end{document}